\RequirePackage{fix-cm}
\let\origvec\vec
\documentclass[smallextended]{svjour3}       

\let\vec\origvec
\smartqed  
\usepackage [latin1]{inputenc}
\usepackage{amsmath}
\usepackage{mathrsfs}
\usepackage{amssymb}
\usepackage{graphicx}
\usepackage{lineno}
\usepackage[justification=centering]{caption}
\usepackage{color}
 \numberwithin{equation}{section}

\journalname{myjournal}
\begin{document}
\title{Periodic travelling interfacial electrohydrodynamic waves: bifurcation and secondary bifurcation}


\author{Guowei Dai \and Fei Xu \and Yong Zhang
}


\institute{Guowei Dai \at
              School of Mathematical Sciences, Dalian University of Technology, Dalian, 116024, People's Republic of China \\
              \email{daiguowei@dlut.edu.cn.}           
              \and
            Fei Xu \at
             School of Mathematical Sciences, Jiangsu University, Zhenjiang 212013, People's Republic of China \\
              \email{xufeiujs@126.com}
        \and
              Yong Zhang (Corresponding author) \at
             Institute of Applied System Analysis, Jiangsu University, Zhenjiang 212013, People's Republic of China \\
              \email{zhangyong@ujs.edu.cn}
}

\date{Received: date / Accepted: date}

\maketitle

\begin{abstract}
In this paper, two-dimensional periodic capillary-gravity waves travelling under the effect of a vertical electric field are considered. The full system is a nonlinear, two-layered and free boundary problem. The interface dynamics arises from the coupling between the Euler equations for the lower fluid layer and an electric contribution from the upper gas layer. To investigate the electrohydrodynamic wave interactions, we first introduce the naive flattening technique to transform the free boundary problem into a fixed boundary problem. Then we prove the existence of the small-amplitude electrohydrodynamic waves with constant vorticity $\gamma$ by using local bifurcation theory. Moreover, we prove that these electrohydrodynamic waves are formally stable in linearized sense. Furthermore, we obtain a secondary bifurcation curve that emerges from the primary branch at a nonlaminar solution as $E_0$ being close to some special value. This secondary bifurcation curve consists of ripples solutions on the interface of a conducting fluid under normal electric fields. As far as we know, this new phenomenon in electrohydrodynamics (EHD) is first established mathematically. It is worth noting that the electric field $E_0$ plays a key role to control the shapes and types of waves on the interface.
\end{abstract}



\section{\bf Introduction}
Water waves propagating on the interface between two fluids have been studied intensively via using either analytical or numerical methods. Many different mathematical methods have been introduced to study the steady or time-dependent solutions both in shallow and deep waters \cite{Vanden}. These waves are usually created by the presence of two different layers of different density combined with a certain configuration of current. As far as we know, the first small-amplitude interfacial periodic traveling capillary-gravity waves on finite depth was constructed by Amick and Turner \cite{AmickT}, where the vorticity of fluid was ignored. When vorticity is included, it was unclear for long time how to leave the regime of small perturbations of configurations with a flat surface. The breakthrough in this direction is due to Constantin and Strauss \cite{Constantin,ConstantinS}, who utilised the semi-hodograph transformation of Dubreil-Jacotin \cite{DubreilJ}. Based on this, the work \cite{WalshOS} allowed for the vorticity in the interfacial wind waves without considering stagnation pionts. Recently, Ambrose, Strauss, and Wright \cite{AmbroseSW} considered the new coordinates and applied the Rabinowitz's global bifurcation theorem to prove the existence of pure capillary interfacial waves and gravity-capillary interfacial waves. We refer to \cite{ChuDE,Sinambela,Wheeler} for more information on two-layer or many-layer density stratified water waves. We also would like to mention the work \cite{JonesT85,MartinM}, where they studied a pseudodifferential equation by finding the secondary bifurcation branches emerging from the primary curves for capillary-gravity water waves without or with vorticity.

In the presence of electric fields, this topic is called Electrohydrodynamics (EHD), which has recently attracted much attention.
It is easy to artificially manipulate considerably strong electric fields with modern engineering techniques to cause significant changes in fluid motion, usually manifested via modifications of the gas-liquid interface dynamics. As a result, EHD enjoys wide industrial applications in chemical engineering, e.g., coating processes, cooling systems of conducting fluids, electrospray technology, etc. (see \cite{MelcherT,Papageorgiou} for a comprehensive overview). Due to the important role that EHD plays in the engineering community, an in-depth understanding of the mathematics behind the scene remains essential. As far as we know, the travelling waves propagating in two-dimensional space under the effects of gravity, surface tension, and electric field have been studied intensively by different authors, either with weakly nonlinear models \cite{Gleeson,HuntD,Wang} or with numerical methods \cite{Doak,DoakG,Gao,LinZW}. However, there have been, so far, few studies fully nonlinear water waves with vorticity under the influence of electric fields by using the analytical method except the work \cite{Smit}, where authors considered singularities in the Electrohydrodynamic equations with ignoring surface tension for simplicity.

In this work, we investigate the resonances mentioned above in the electrified Euler equation with surface tension by using the bifurcation theory. The novelty here is however that we remove the assumption of irrotationality in the case of electrohydrodynamics with an assumption of globally constant vorticity in the lower fluid. The other contribution of the paper is to reveal the effect of electric field $E_0$ on the appearance of different types of electrohydrodynamic waves. Especially, we prove that there would be periodic electrohydrodynamic waves if (\ref{eq3.16}) holds and establish the existence of ripples for the first time on the interface as $E_0$ attains some value, which may be useful in physical and industrial applications. Our results are true for electrohydrodynamic capillary as well. A very delicate issue is to reformulate the problem such that it becomes amenable to a certain local bifurcation theorem in order to construct these waves of small amplitude. To this end and what may seem surprisingly, the flattening transform plays a key role. It is known that this transformation can admit the occurrence of stagnation points and critical layers inside fluid \cite{EhrnstromEW,HenryBM,Wahlen09}. However, this is an interesting and important question but beyond the scope of the preliminary investigations here.

The rest of the paper is structured as follows. The problem is mathematically formulated in Section 2. We prove the existence and stability of electrohydrodynamic periodic capillary-gravity waves by using the celebrated Crandall-Rabinowitz theorems \cite{CrandallR,CrandallR1} in Section 3. Finally, the ripples are shown by applying the secondary bifurcation theorem \cite{Shearer} in Section 4.

\section{\bf Preliminaries}
\subsection{\bf Introduction of the problem}
We consider a two-dimensional steady, inviscid, incompressible, and perfectly conducting fluid of constant density $(\rho=1)$ and finite depth
bounded below by a flat electrode. We first introduce a Cartesian coordinates $(x,y)$ such that $x$-axis points to the horizontal and $y$-axis points to the vertical. Assume that the flat electrode is given by $y=-d$ with $d>0$, the free interface is given by $y=\eta(x)$ and find the function $\psi(x,y)$ called the stream function, providing the velocity field $(\psi_y, -\psi_x)$ of the fluid, which satisfies the following equations and boundary conditions:
\begin{eqnarray}
\left\{\begin{array}{llll}{\Delta\psi=\gamma} & {\text { for }-d < y < \eta(x)}, \\
{\psi=0} & {\text { on } ~y= \eta(x)}, \\
{\psi=m} & { \text { on }~  y =-d},\end{array}\right. \label{eq2.1}
\end{eqnarray}
where $\gamma$ is the constant vorticity meaning that the flow is to be rotational and the constant $m$ is the relative
mass flux defined by $m=\int_{-d}^{\eta(x)}-\psi_y \,dy$.

In addition, we assume that the fluid is perfectly conducting so that the electric strength is zero within the fluid. The surrounding medium,
which occupies the region above the liquid, is assumed to be dielectric with permittivity $\epsilon_0$. Its density is very small and negligible in comparison to that of the conducting fluid. The upper layer is also bounded by a  flat electrode $y =d$. The electrostatic limit of Maxwell's equation implies that the induced magnetic fields are negligible, and it then follows that the electric field is irrotational due to Faraday's law. Therefore, we can introduce the voltage potential $V$, such that the vertical electric field above the liquid surface satisfies $\mathbf{E}=\nabla V$ with $\mathbf{E}=(0,E_0)$ near the above electrode $y=d$ for some constant $E_0$, and hence the voltage potential satisfies
\begin{equation}\label{eq2.2}
\Delta V=0\quad \text{for} \quad \eta(x)<y<d.
\end{equation}
Due to the conducting nature of the fluid, the voltage potential $V$ is invariant in the fluid domain and chosen to be zero without loss of
generality. It is easy to see that $V=V_0$ at the top electrode, where $V_0=E_0 d$ is a constant. Then the boundary conditions for $V$ are imposed by
\begin{eqnarray}
\left\{\begin{array}{llll}{V=0} & {\text { for } y = \eta(x)}, \\
{V=V_0} & { \text { for } ~ y=d}. \end{array}\right. \label{eq2.3}
\end{eqnarray}
Although the voltage potential vanishes on the free surface, the normal component of the electric field gives rise to normal stress. Therefore, the electric field and the fluid motion are coupled through the Maxwell stress tensor leading to the following modified Bernoulli law
\begin{eqnarray}
|\nabla\psi|^2+2g\eta-\epsilon_0|\nabla V|^2-2\sigma\frac{\eta_{xx}}{\left(1+\eta_x^2\right)^\frac{3}{2}}=Q~~\text{on}~~y=\eta(x), \label{eq2.4}
\end{eqnarray}
where $g$ is the gravitational constant of acceleration, $\sigma$ is the coefficient of interface tension and $Q$ is the Bernoulli's constant. A schematic of the problem is presented in Fig. \ref{fig1}.
\begin{figure}[ht]
\centering
\includegraphics[width=0.75\textwidth]{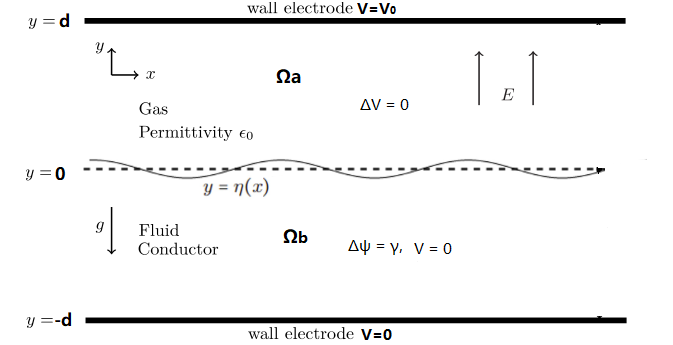}
\caption{The schematic of the problem.}
\label{fig1}
\end{figure}

As a first step, we introduce a constant $\lambda$ into the problem which will allow us to describe
the trivial solutions of (\ref{eq2.1})-(\ref{eq2.4}), that is to say, the solutions with a flat wave profile located at $y=0$.
Later on, we use this constant as a bifurcation parameter to find nontrivial solutions of (\ref{eq2.1})-(\ref{eq2.4}).
To this end, when $\eta(x)=0$ and $(V,\psi)$ is independent of $x$, we observe that the functions $(\overline{V},\overline{\psi})$ solving (\ref{eq2.1})-(\ref{eq2.4}) are  given by
\begin{eqnarray}
(\overline{V},\overline{\psi})=(E_0 y, \frac{\gamma y^2}{2}+\lambda y) \label{eq2.5}
\end{eqnarray}
provided the constants $m$ and $Q$ satisfy
$$
m=m(\lambda)=\frac{\gamma d^2}{2}-\lambda d,\quad Q=Q(\lambda)=\lambda^2-\epsilon_0 E_{0}^2.
$$
The functions $(\overline{V},\overline{\psi})$ obtained in (\ref{eq2.5}) will be taken as the trivial solutions with velocity field $(\overline{\psi}_y, -\overline{\psi}_x)=(\gamma y+\lambda,0)$. Thus, the bifurcation parameter $\lambda$ represents physically the horizontal wave speed of the electrohydrodynamic waves on the interface $y=0$.
\subsection{\bf Reformulation via the naive flattening transform}

The main difficulties associated with the system (\ref{eq2.1})-(\ref{eq2.4}) are its nonlinear character and the fact that the interface $y=\eta(x)$ is unknown. The latter difficulty can be overcome by introducing a suitable naive flattening transform. Let's first define the unknown domains by
$$
\Omega_{a}:=\{ (x,y)\in \mathbb{R}^2: \eta(x)<y<d \}, \quad \Omega_{b}:=\{ (x,y)\in \mathbb{R}^2: -d<y<\eta(x) \}
$$
and the following horizontal strips by
$$
D_{a}:=\{ (q,p)\in \mathbb{R}^2: 0<p<d \}, \quad D_{b}:=\{ (q,p)\in \mathbb{R}^2: -d<p<0 \}
$$
The main idea is to find a transform to map the unknown domains $\Omega_a$ and $\Omega_b$ into the strips $D_a$ and $D_b$.
\begin{figure}[ht]
\centering
\includegraphics[width=0.75\textwidth]{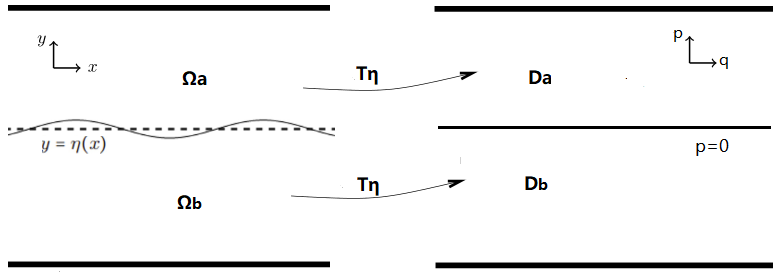}
\caption{The flatness of the fluid domains.}
\label{fig2}
\end{figure}

Here we introduce the naive flatting transform by
\begin{eqnarray}
T_{\eta}: (x,y)\mapsto(q,p)=\left\{\begin{array}{llll}{\left(x, y-\frac{\eta(x)(d-y)}{d-\eta(x)}\right)} & {\text { for }~\eta(x) \leq y < d}, \\
~\\
{\left(x, y-\frac{\eta(x)(d+y)}{d+\eta(x)}\right)} & { \text { for }~-d< y \leq \eta(x)}.\end{array}\right. \label{eq2.6}
\end{eqnarray}
Assume $|\eta(x)|<d$, it is easy to check that the naive flatting transform $T_{\eta}$ is a diffeomorphism and maps $\Omega_a$ and $\Omega_b$ into $D_a$ and $D_b$ as in Fig \ref{fig2}.

In the coordinates $(q,p)$, then the equation (\ref{eq2.1}) of $\psi$ would become
\begin{equation}\label{eq2.7}
\left\{\begin{array}{llll}
{\psi_{qq}-\frac{2(d+p)\eta'}{d+\eta}\psi_{qp}+\frac{(d+p)^2\eta'^2+d^2}{(d+\eta)^2}\psi_{pp}-\frac{(d+p)\left((d+\eta)\eta''-2\eta'^2\right)}{(d+\eta)^2}\psi_p=\gamma} &~{p\in(-d,0)},\\
{\psi=0} & ~\text{on}~~{p=0},\\
{\psi=m} &~\text{on}~~{p=-d},\\
\end{array}\right.
\end{equation}
the equations (\ref{eq2.2})-(\ref{eq2.3}) of $V$ would be
\begin{equation}\label{eq2.8}
\left\{\begin{array}{llll}
{V_{qq}-\frac{2(d-p)\eta'}{d-\eta}V_{qp}+\frac{(d-p)^2\eta'^2+d^2}{(d-\eta)^2}V_{pp}-\frac{(d-p)\left((d-\eta)\eta''+2\eta'^2\right)}{(d-\eta)^2}V_p=0} &~{p\in(0,d)},\\
{V=0} & ~\text{on}~~{p=0},\\
{V=E_0d} &~\text{on}~~{p=d}\\
\end{array}\right.
\end{equation}
and the equation (\ref{eq2.4}) on the interface will take the form
\begin{eqnarray}\label{eq2.9}
& & \psi_q^2+\frac{d^2(\eta'^2+1)}{(d+\eta)^2}\psi_p^2-\frac{2d\eta'}{d+\eta}\psi_q\psi_p+2g\eta-2\sigma\frac{\eta''}{\left(1+\eta'^2\right)^\frac{3}{2}}\nonumber\\
& &-\epsilon_0\left(V_q^2+\frac{d^2(\eta'^2+1)}{(d-\eta)^2}V_p^2-\frac{2d\eta'}{d-\eta}V_qV_p\right)-Q=0~~\text{on}~~p=0.
\end{eqnarray}



\section{The dynamics of periodic electrohydrodynamic waves}
In this section, we aim to construct the nontrivial solutions of (\ref{eq2.7})-(\ref{eq2.9}) by applying bifurcation theory. Thus
it is necessary to transform the problem (\ref{eq2.7})-(\ref{eq2.9}) into a suitable abstract operator equation.

We first introduce a functional analytic setup. Without loss of generality, here we choose the constant $d=1$. Define the function space for an arbitrary $\alpha\in (0,1)$ by
 $$X:=\{\eta\in C^{2,\alpha}_{per,e,0}(\mathbb{R}): |\eta|<1\},$$
where $C^{k,\alpha}_{per}$ denotes the space of functions in  H\"{o}lder class with the variable $q$
being periodic, even and having zero average. Similarly, we introduce the space
$$Y:=\{\eta\in C^{0,\alpha}_{per,e,0}(\mathbb{R}):  |\eta|<1\},$$

It follows from \cite[Theorem 6.14]{Gilbarg} that there exist a unique solution $\psi\in C^{2,\alpha}_{per,e}(\overline{D_b})$ to the problem (\ref{eq2.7}) and a unique solution $V\in C^{2,\alpha}_{per,e}(\overline{D_a})$ to the problem (\ref{eq2.8}) for any $\eta\in X$. Therefore, for a fixed $\eta$, we have that
$$
V:=V(E_0, \eta),\quad \psi:=\psi(\lambda,\eta).
$$
In addition, $V(E_0,\eta)$ and $\psi(\lambda,\eta)$ are differential with respect to $\eta$ by the implicit function theorem. In fact, it follows \cite[Lemma 4.1]{HenryBM} that $V\in C^2(\mathbb{R}\times X, C^{2,\alpha}_{per,e}(\overline{D_a}))$ and $\psi\in C^2(\mathbb{R}\times X, C^{2,\alpha}_{per,e}(\overline{D_b}))$. Taking $V(E_0,\eta)$ and $\psi(\lambda,\eta)$ into the equation (\ref{eq2.9}) on the interface allows us to define a new abstract equation as our bifurcation equation by
\begin{eqnarray} \label{eq3.1}
F(\lambda,E_0,\eta): & & =\psi_q^2+\frac{\eta'^2+1}{(1+\eta)^2}\psi_p^2-\frac{2\eta'}{1+\eta}\psi_q\psi_p+2g\eta-2\sigma\frac{\eta''}{\left(1+\eta'^2\right)^\frac{3}{2}}\nonumber\\
& &-\epsilon_0\left(V_q^2+\frac{\eta'^2+1}{(1-\eta)^2}V_p^2-\frac{2\eta'}{1-\eta}V_qV_p\right)\bigg|_{p=0}-Q=0,
\end{eqnarray}
where $F: \mathbb{R}^2\times X\rightarrow Y$ and $V=V(E_0,\eta), \psi=\psi(\lambda,\eta)$.

\subsection{\bf Spectrum of the linearized operator as a Fourier multiplier }
In order to prove the occurrence of local bifuraction, we will deduce spectral properties of the linearized operator $\partial_{\eta}F(\lambda, E_0, 0)$. Indeed, we can prove that $\partial_{\eta}F(\lambda, E_0, 0)$ is a Fourier multiplier, which allows us to analyse its Fredholm properties by examining the behaviour of its symbols $D(\lambda, E_0)$. Recalling that $V(E_0,0)=\overline{V}$ and $\psi(\lambda,0)=\overline{\psi}$ where $\overline{V}$ and $\overline{\psi}$ are given by (\ref{eq2.5}), we have that
\begin{eqnarray} \label{eq3.2}
F_{\eta}(\lambda,E_0,0)[\eta]& & =
\partial_{\eta}F(\lambda,E_0,0)[\eta]+\partial_{\psi}F(\lambda,E_0,0)[\partial_{\eta}\psi(\lambda,0)[\eta]]\nonumber\\
& &+\partial_{V}F(\lambda,E_0,0)[\partial_{\eta}V(\lambda,0)[\eta]]\bigg|_{p=0},
\end{eqnarray}
for all $\eta\in X$. Define
\begin{eqnarray}\label{eq3.3}
w(q,p):=\partial_{\eta}\psi(\lambda,0)[\eta],\qquad h(q,p):=\partial_{\eta}V(\lambda,0)[\eta].
\end{eqnarray}
It's easy to see that
\begin{eqnarray}\label{eq3.4}
\partial_{\eta}F(\lambda,E_0,0)[\eta]=-2\overline{\psi}_{p}^2\eta+2g\eta-2\sigma\eta''-2\epsilon_0\overline{V}_{p}^2\eta
\end{eqnarray}
and
\begin{eqnarray}\label{eq3.5}
\partial_{\psi}F(\lambda,E_0,0)[w]=2\overline{\psi}_{p}w_p,\qquad \partial_{V}F(\lambda,E_0,0)[h]=-2\epsilon_0\overline{V}_{p}h_p.
\end{eqnarray}

On the other hand, it follows from (\ref{eq2.5}) that
$$
\overline{V}=E_0 y=E_0(p(1-\eta)+\eta )
$$
and
 $$
\overline{\psi}=\frac{\gamma y^2}{2}+\lambda y=\frac{\gamma}{2}\left(p(1+\eta)+\eta\right)^2+\lambda\left(p(1+\eta)+\eta\right)
$$
Thus, (\ref{eq3.4}) and (\ref{eq3.5}) would become that
\begin{eqnarray}\label{eq3.6}
\partial_{\eta}F(\lambda,E_0,0)[\eta]=-2\lambda^2\eta+2g\eta-2\sigma\eta''-2\epsilon_0E_{0}^2\eta
\end{eqnarray}
and
\begin{eqnarray}\label{eq3.7}
\partial_{\psi}F(\lambda,E_0,0)[w]=2\lambda w_p,\qquad \partial_{V}F(\lambda,E_0,0)[h]=-2\epsilon_0E_0h_p.
\end{eqnarray}

Based on (\ref{eq3.2})-(\ref{eq3.3}) and (\ref{eq3.6})-(\ref{eq3.7}), it remains for us to determine $w_{p}(q,0)$ and $h_{p}(q,0)$. By observation, differentiating the equations of (\ref{eq2.7}) with respect to $\eta$, then $w=\partial_{\eta}\psi(\lambda,0)[\eta]$ is the solution of the following problem
\begin{equation}\label{eq3.8}
\left\{\begin{array}{llll}
{\Delta w=2\overline{\psi}_{pp}\eta+(1+p)\overline{\psi}_{p}\eta''=2\gamma\eta+(1+p)(\gamma p+\lambda)\eta''} &~{p\in(-1,0)},\\
{w=0} & ~\text{on}~~{p=0},\\
{w=0} &~\text{on}~~{p=-1}.\\
\end{array}\right.
\end{equation}
Similarly, differentiating the equations of (\ref{eq2.8}) with respect to $\eta$, then we have that $h=\partial_{\eta}V(E_0,0)[\eta]$ is the solution of the following problem
\begin{equation}\label{eq3.9}
\left\{\begin{array}{llll}
{\Delta h=-2\overline{V}_{pp}\eta+(1-p)\overline{V}_{p}\eta''=(1-p)E_0\eta''} &~{p\in(0,1)},\\
{h=0} & ~\text{on}~~{p=0},\\
{h=0} &~\text{on}~~{p=1}.\\
\end{array}\right.
\end{equation}

Since $\eta\in X$ and (\ref{eq3.8})-(\ref{eq3.9}), we consider now the Fourier series expansions of $\eta,w$ and $h$ by
\begin{eqnarray}
\eta(q)=\sum_{k=1}^{\infty}\eta_k\cos(kq), ~w=\sum_{k=1}^{\infty}\eta_kw_{k}(p)\cos(kq),~h=\sum_{k=1}^{\infty}\eta_kh_k(p)\cos(kq). \nonumber
\end{eqnarray}
Then the Fourier coefficients $w_k(p)$ of the solution $w$ to (\ref{eq3.8}) can be written as
\begin{eqnarray}\label{eq3.10}
w_k(p)=\int^0_{-1} G_1(p,r)\left( 2\gamma-(1+r)(\gamma r+\lambda)k^2 \right) \,dr
\end{eqnarray}
where
\begin{eqnarray}
G_1(p,r)=\frac{1}{k\sinh(k)}\left\{\begin{array}{llll}{\sinh(k(1+p))\sinh(kr)} & {\text { for }~p\leq r}, \\
~\\
{\sinh(kp)\sinh(k(1+r))} & { \text { for }~p\geq r}\end{array}\right. \nonumber
\end{eqnarray}
is Green's function of the operator $\frac{\partial}{\partial p^2}-k^2$ on the interval $[-1,0]$ with Dirichlet boundary conditions (see Subsection 5.1).
Similarly, the Fourier coefficients $h_k(p)$ of the solution $h$ to (\ref{eq3.9}) can be written as
\begin{eqnarray}\label{eq3.11}
h_k(p)=\int^1_0 G_2(p,r)E_0(r-1)k^2 \,dr
\end{eqnarray}
where
\begin{eqnarray}
G_2(p,r)=-\frac{1}{k\sinh(k)}\left\{\begin{array}{llll}{\sinh(kp)\sinh(k(1-r))} & {\text { for }~p\leq r}, \\
~\\
{\sinh(k(1-p))\sinh(kr)} & { \text { for }~p\geq r}\end{array}\right. \nonumber
\end{eqnarray}
is Green's function of the operator $\frac{\partial}{\partial p^2}-k^2$ on the interval $[0,1]$ with Dirichlet boundary conditions (see Subsection 5.1).

Thus, it follows from (\ref{eq3.10}) and (\ref{eq3.11}) that
\begin{eqnarray} \label{eq3.12}
w_{k,p}(0)& & =
\frac{1}{\sinh(k)}\int_{-1}^{0}(2\gamma-(1+p)(\gamma p+\lambda)k^2)\sinh(k(1+p))dp\nonumber\\
& &=\frac{(\gamma+\lambda-\lambda k)e^k-(\gamma+\lambda+\lambda k)e^{-k}}{2\sinh(k)}
\end{eqnarray}
and
\begin{eqnarray} \label{eq3.13}
h_{k,p}(0)& & =
-\frac{1}{\sinh(k)}\int_{0}^{1}E_0(p-1)k^2\sinh(k(1-p))dp \nonumber\\
& &=\frac{(k-1)e^k+(k+1)e^{-k}}{2\sinh(k)}E_0.
\end{eqnarray}

Summarising, we find that the Fr\'{e}chet derivative of $\partial_{\eta}F(\lambda,E_0,0)$ is the Fourier multiplier, that is to say
\begin{eqnarray} \label{eq3.14}
\partial_{\eta}F(\lambda,E_0,0)\sum_{k=1}^{\infty}\eta_{k}\cos(kq)=\sum_{k=1}^{\infty}D_{k}(\lambda,E_0)\eta_{k}\cos(kq)
\end{eqnarray}
where
\begin{eqnarray}
D_{k}(\lambda,E_0):& & =-2\lambda^2+2g+2\sigma k^2-2\epsilon_0E_0^2-\epsilon_0E_0^2\frac{(k-1)e^k+(k+1)e^{-k}}{\sinh(k)}\nonumber\\
& &+\lambda\frac{(\gamma+\lambda-\lambda k)e^k-(\gamma+\lambda+\lambda k)e^{-k}}{\sinh(k)} \nonumber\\
& &=-\frac{2k}{\tanh(k)}\left(\lambda^2-\frac{\gamma}{k}\tanh(k)\lambda+\epsilon_0E_0^2-\frac{g+\sigma k^2}{k}\tanh(k)\right). \nonumber
\end{eqnarray}
Indeed, this formula
is called dispersion relation. Once the dispersion relation vanishes, which means the linearized operator $\partial_{\eta}F(\lambda,E_0,0)$ is degenerate and the nontrivial solutions may occur.
It is obvious that
\begin{eqnarray} \label{eq3.15}
\lambda^*_{k,\pm}=\frac{\gamma}{2k}\tanh(k)\pm\sqrt{\frac{\gamma^2\tanh^2(k)}{4k^2}+\frac{(g+\sigma k^2)\tanh(k)}{k}-\epsilon_0E_0^2}
\end{eqnarray}
are solutions to $D_{k}(\lambda,E_0)=0$ if there holds that
\begin{eqnarray} \label{eq3.16}
\frac{(g+\sigma k^2)\tanh(k)}{k}>\epsilon_0E_0^2, ~~\text{for any}~~k\in\mathbb{N}^+.
\end{eqnarray}

\begin{remark} \label{rem1}
In fact, the wave number $k$ and the electric field $E_0$ satisfying
$$\frac{\gamma^2\tanh^2(k)}{4k^2}+\frac{(g+\sigma k^2)\tanh(k)}{k}>\epsilon_0E_0^2$$
are enough to conclude the occurrence of electrohydrodynamic waves, which means physically that the electric field is not too strong. However, a slightly stronger assumption (\ref{eq3.16}) would ensure that
$$
\lambda^*_{k,+}>0,\qquad \lambda^*_{k,-}<0,
$$
which play a key role in the ripples. That is to say that we will let $\lambda^*_{k,+}=\lambda^*_{l,+}$ (or $\lambda^*_{k,+}=\lambda^*_{l,+}$) for any $k\neq l$ in Section 4, which leads to $E_0^2=E_{k,l}$ where $E_{k,l}$ is defined by (\ref{eq4.5}) and the case of bifurcation from two-dimensional kernels would occur.
\end{remark}

\subsection{\bf The existence of local primary bifurcation}

In this subsection, we will give our first main result on the existence of  electrohydrodynamic capillary-gravity waves by using Theorem \ref{thm6.1} in Appendix.

Based on the argument in Subsection 3.1, it is obvious that
\begin{eqnarray} \label{eq3.17}
F(\lambda,E_0,0)=0
\end{eqnarray}
for any $\lambda\in \mathbb{R}$, which implies that (H1) in Theorem \ref{thm6.1} holds. It follows from (\ref{eq3.14}) and (\ref{eq3.15}) that the bounded linear operator $\partial_{\eta}F(\lambda,E_0,0):\mathbb{R}^2\times X\rightarrow Y$ is invertible whenever $\lambda\neq \lambda^*_{k,\pm}$ for any integer $k\geq 1$. By the implicit function theorem, these points are not bifurcation points.

However, we claim that (H2) in Theorem \ref{thm6.1} holds for each $\lambda^*\in \{\lambda^*_{k,\pm}: k\in \mathbb{N}^+\}$. From (\ref{eq3.14}) and (\ref{eq3.15}), it is obvious that $\mathcal{N}(\partial_{\eta}F(\lambda^*,E_0,0))$ is one-dimensional and generated by $\eta^*=\cos(kq)\in X$.
While $\mathcal{R}(\partial_{\eta}F(\lambda^*,E_0,0))$ is the closed subspace of $Y$ formed by the elements $f\in Y$ satisfying that
$$
\int^{\pi}_{-\pi}f(q)\cos(kq)dq=0.
$$
Then it follows from (\ref{eq3.14})-(\ref{eq3.16}) that
\begin{equation} \label{eq3.18}
\partial_{\lambda\eta}F(\lambda^*,E_0,0)[1,\eta^*]=-\frac{2k}{\tanh(k)}\left(2\lambda^*-\frac{\gamma}{k}\tanh(k)\right)\eta^*\notin \mathcal{R}(\partial_{\eta}F(\lambda^*,E_0,0)),
\end{equation}
which shows that (H2) holds for every $\lambda^*\in \{\lambda^*_{k,\pm}: k\in \mathbb{N}^+\}$. With (\ref{eq3.17}) and (\ref{eq3.18}) in hands, then we can obtain the following result by applying the Theorem \ref{thm6.1} in Appendix.

\begin{theorem} \label{thm3.1} (Primary branches of bifurcation)
Assume that the electric field $E_0$ such that (\ref{eq3.16}) holds and $E_0^2\neq E_{k,l}$ for any positive integer $k,l$ and the vorticity function $\gamma\in \mathbb{R}$, let $\lambda_{k,\pm}^*$ be shown as in (\ref{eq3.15}). For any $\lambda\in \mathbb{R}\backslash \left\{\lambda_{k,\pm}^*:k\in \mathbb{N}^+\right\}$, the
solutions of (\ref{eq3.1}) are trivial. For $\lambda\in \lambda_{k,\pm}^*$ and each choice of sign $\pm$, there exists in the space $\mathbb{R}\times X$ a
continuous curve $\mathcal{K}_{k,\pm} = \left\{\left(\lambda_{k,\pm}(s), \eta(s)\right) : s\in \mathbb{R}\right\}$
of solutions of (\ref{eq3.1}) satisfying

(i) $\left(\lambda_{k,\pm}(0), \eta(0)\right)=\left(\lambda_{k,\pm}^*, 0\right)$;

(ii) $\eta=\eta(s)$ in $X$ with $\eta(s)= s\eta^*+o(s)$ for $|s|<\varepsilon$, where $\eta^*=\cos(kq)$ and $\varepsilon>0$ sufficiently
small;

(iii) there exist a neighbourhood $\mathcal{U}_{k,\pm}$ of $\left(\lambda_{k,\pm}^*, 0\right)$ in $\mathbb{R}\times X$ and
$\varepsilon>0$ sufficiently small such that
\begin{equation}
\left\{\left(\lambda_{k,\pm}, \eta\right)\in \mathcal{U}_{k,\pm} : \eta\not\equiv 0\,\, \text{and}\,\, f(\lambda,\eta)=0\right\} = \left\{\left(\lambda_{k,\pm}(s), \eta(s)\right) : 0 < |s| < \varepsilon\right\}.\nonumber
\end{equation}
\end{theorem}

\subsection{\bf The local stability of the primary bifurcation}
Inspired by \cite{ConstantinS1,DaiLZ}, we now establish the local stability of the electrohydrodynamic waves lying on the primary branch $\mathcal{K}_{k,\pm}$ obtained in Theorem \ref{thm3.1}. We would like to mention that the concept of formal stability has been given in Appendix.

\begin{theorem} \label{thm3.2} (The stability of the local primary bifurcation curve)\\
(i) For the fixed  primary bifurcation branch $(\lambda_{k,\pm}(s),\eta(s))$ emanating from $\lambda_{k,\pm}^*$, we have that
$$\lambda_{k,\pm}'(0)=0$$
and there holds that
$$
\lambda_{k,+}''(0)<0,~~~\lambda_{k,-}''(0)>0
$$
when the constant vorticity $\gamma\in[-\varepsilon,\varepsilon]$ for $\varepsilon$ small enough.\\
(ii) As the constant vorticity $\gamma\in[-\varepsilon,\varepsilon]$, near the bifurcation point $\lambda_{k,+}^*$, the laminar solution is formally stable for $\lambda>\lambda_{k,+}^*$ but unstable for $\lambda<\lambda_{k,+}^*$. However, near the bifurcation point $\lambda_{k,-}^*$, the laminar solution is unstable for $\lambda>\lambda_{k,-}^*$ but formally stable for $\lambda<\lambda_{k,-}^*$.\\
(iii) As the constant vorticity $\gamma\in[-\varepsilon,\varepsilon]$, the nontrivial solution curves are formally stable near the bifurcation points $\lambda_{k,\pm}^*$ in local. (see Fig.3)
\end{theorem}
\begin{figure}[ht] \label{fig3}
\centering
\includegraphics[width=0.75\textwidth]{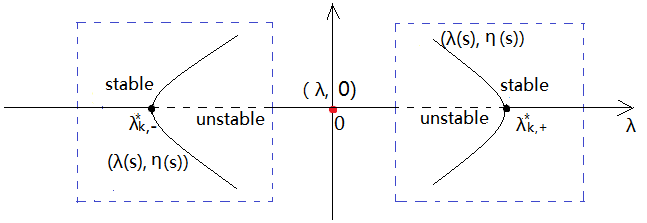}
\centering
\caption{The local stability near bifurcation points $\lambda_{k,\pm}^*$.}
\end{figure}
\begin{proof}
Let us first use the Theorem \ref{thm6.2} to prove the bifurcation direction (i).
For any $\eta\in X$, define
\begin{eqnarray}
\widetilde{l}(\eta)=\frac{1}{\pi}\int_{-\pi}^{\pi} \eta\cos(kq) dq.\nonumber
\end{eqnarray}
Then $\widetilde{l}$ is a linear functional on $X$ such that $\widetilde{l}\left(\cos(kq)\right)=1$.
It is clearly $\mathcal{R}(F_{\eta}\left(\lambda_{k,\pm}^*,0\right))=\left\{\eta:\widetilde{l}(\eta)=0\right\}$.
Due to $\eta^*=\cos(kq)$, it follows (\ref{eq3.18}) that
\begin{equation}\label{eq3.19}
\left\langle \widetilde{l}, F_{\lambda\eta}\left(\lambda_{k,\pm}^*,0\right)[1,\eta^*]\right\rangle=\mp\frac{4k}{\tanh(k)}\sqrt{\frac{\gamma^2\tanh^2(k)}{4k^2}+\frac{(g+\sigma k^2)\tanh(k)}{k}-\epsilon_0E_0^2}.
\end{equation}
Let $T_{k}:=\frac{\tanh(k)}{k}$ for $k\in\mathbb{N}^+$, it follows from (\ref{eq5.4}) that
\begin{eqnarray} \label{eq3.20}
~& &F_{\eta\eta}(\lambda_{k,\pm}^*, 0)[\eta^*,\eta^*]
~\nonumber\\
& &=2(\lambda_{k,\pm}^{*2}-\epsilon_0E_0^{2}) k^2\sin^2(kq)+6(\lambda_{k,\pm}^{*2}-\epsilon_0E_0^{2})\cos^2(kq)     \nonumber\\
& &-2\epsilon_0 \frac{(1-T_k)^2}{T_k^2}E_0^{2}\cos^2(kq)+2 \frac{\left((\gamma+\lambda_{k,\pm}^*)T_k-\lambda_{k,\pm}^*\right)^2}{T_k^2}\cos^2(kq)  \nonumber\\
& & -8\epsilon_0E_0^{2}\frac{1-T_k}{T_k}\cos^2(kq)     -8 \lambda_{\pm}^*\frac{(\gamma+\lambda_{k,\pm}^*)T_k-\lambda_{k,\pm}^*}{T_k}\cos^2(kq)
\end{eqnarray}
which leads to that
\begin{eqnarray}\label{eq3.21}
\left\langle \widetilde{l}, F_{\eta\eta}(\lambda_{k,\pm}^*, 0)[\eta^*,\eta^*]\right\rangle=0,
\end{eqnarray}
where we use the following facts
$$
\int^{\pi}_{-\pi} \cos^3(kq)dq=0, \quad \int^{\pi}_{-\pi} \sin^2(kq) \cos(kq)dq=0.
$$
Combining (\ref{eq3.19}) with (\ref{eq3.21}) and applying Theorem \ref{thm6.2}, we obtain that
\begin{equation}
\lambda_{k,\pm}'(0)=-\frac{\left\langle \widetilde{l}, F_{\eta\eta}(\lambda_{k,\pm}^*, 0)[\eta^*,\eta^*]\right\rangle}{2\left\langle \widetilde{l}, F_{\lambda\eta}\left(\lambda_{k,\pm}^*,0\right)[1,\eta^*]\right\rangle} =0,\nonumber
\end{equation}
which indicates that the primary bifurcation is pitchfork bifurcation.

It follows from (\ref{eq3.12})-(\ref{eq3.13}) that
\begin{eqnarray}
w_{k,p}(0)=\frac{(\gamma+\lambda)T_k-\lambda}{T_k}  ,\quad h_{k,p}(0)=\left(\frac{1-T_k}{T_k}\right)E_0. \nonumber
\end{eqnarray}
Combining this with (\ref{eq5.6}), we can easily deduce that
\begin{eqnarray} \label{eq3.22}
~& &F_{\eta\eta\eta}(\lambda_{k,\pm}^*, 0)[\eta^*,\eta^*,\eta^*]
~\nonumber\\
& &=-12k^2\left(\lambda_{k,\pm}^{*2}+ \epsilon_{0}E_0^{*2} \right)\sin^2(kq)\cos(kq)-24\left(\lambda_{k,\pm}^{*2}+ \epsilon_{0}E_0^{*2} \right)\cos^3(kq)
  \nonumber\\
& &-18k^4\sigma\sin^2(kq)\cos(kq)
+36\lambda_{k,\pm}^*\frac{(\gamma+\lambda_{k,\pm}^*)T_k
-\lambda_{k,\pm}^*}{T_k}\cos^3(kq)
\nonumber\\
& &+12k^2\lambda_{k,\pm}^*\frac{(\gamma+\lambda_{k,\pm}^*)T_k-\lambda_{k,\pm}^*}{T_k}\sin^2(kq)\cos(kq) \nonumber\\
& &-12k^2\epsilon_0E_0^{*2}\left(\frac{1-T_k}{T_k}\right)\sin^2(kq)\cos(kq)-36\epsilon_0E_0^{*2}\left(\frac{1-T_k}{T_k}\right)\cos^3(kq) \nonumber\\
& &-4 \left(\frac{(\gamma+\lambda_{k,\pm}^*)T_k-\lambda_{k,\pm}^*}{T_k}\right)^2\cos^3(kq) -4\epsilon_0\left(\frac{1-T_k}{T_k}\right)^2E_0^{*2}\cos^3(kq).\nonumber\\
\end{eqnarray}
Based on (\ref{eq3.22}), we obtain that
\begin{eqnarray}\label{eq3.23}
&~&\left\langle \widetilde{l},F_{\eta\eta\eta}(\lambda_{k,\pm}^*, 0)[\eta^*,\eta^*,\eta^*]\right\rangle  \nonumber \\
& &=-3k^2\left(\lambda_{k,\pm}^{*2}+ \epsilon_{0}E_0^{*2} \right)-18\left(\lambda_{k,\pm}^{*2}+ \epsilon_{0}E_0^{*2} \right)
-\frac{9}{2}k^4\sigma  \nonumber\\
& &+3k^2\lambda_{k,\pm}^*\frac{(\gamma+\lambda_{k,\pm}^*)T_k-\lambda_{k,\pm}^*}{T_k}+27\lambda_{k,\pm}^*\frac{(\gamma+\lambda_{k,\pm}^*)T_k
-\lambda_{k,\pm}^*}{T_k}
\nonumber\\
& &-3k^2\epsilon_0E_0^{*2}\left(\frac{1-T_k}{T_k}\right)-27\epsilon_0E_0^{*2}\left(\frac{1-T_k}{T_k}\right) \nonumber\\
& &-3 \left(\frac{(\gamma+\lambda_{k,\pm}^*)T_k-\lambda_{k,\pm}^*}{T_k}\right)^2 -3\epsilon_0\left(\frac{1-T_k}{T_k}\right)^2E_0^{*2}.
\end{eqnarray}
Based on (\ref{eq3.23}), as $\gamma=0$, we find that
\begin{eqnarray}
&~&\left\langle \widetilde{l},F_{\eta\eta\eta}(\lambda_{k,+}^*, 0)[\eta^*,\eta^*,\eta^*]\right\rangle  \nonumber \\
& &=-3k^2\left(\lambda_{k,+}^{*2}+ \epsilon_{0}E_0^{*2} \right)-18\left(\lambda_{k,+}^{*2}+ \epsilon_{0}E_0^{*2} \right)
-\frac{9}{2}k^4\sigma  \nonumber\\
& &+3k^2\lambda_{k,+}^*\frac{\lambda_{k,+}^*(T_k-1)}{T_k}+27\lambda_{k,+}^*\frac{\lambda_{k,+}^*(T_k-1)}{T_k}
\nonumber\\
& &-3k^2\epsilon_0E_0^{*2}\left(\frac{1-T_k}{T_k}\right)-27\epsilon_0E_0^{*2}\left(\frac{1-T_k}{T_k}\right) \nonumber\\
& &-3 \left(\frac{\lambda_{k,+}^*T_k-\lambda_{k,+}^*}{T_k}\right)^2 -3\epsilon_0\left(\frac{1-T_k}{T_k}\right)^2E_0^{*2}<0
\end{eqnarray}
and
\begin{eqnarray}
&~&\left\langle \widetilde{l},F_{\eta\eta\eta}(\lambda_{k,-}^*, 0)[\eta^*,\eta^*,\eta^*]\right\rangle  \nonumber \\
& &=-3k^2\left(\lambda_{k,-}^{*2}+ \epsilon_{0}E_0^{*2} \right)-18\left(\lambda_{k,-}^{*2}+ \epsilon_{0}E_0^{*2} \right)
-\frac{9}{2}k^4\sigma  \nonumber\\
& &+3k^2\lambda_{k,-}^*\frac{\lambda_{k,-}^*(T_k-1)}{T_k}+27\lambda_{k,-}^*\frac{\lambda_{k,-}^*(T_k-1)}{T_k}
\nonumber\\
& &-3k^2\epsilon_0E_0^{*2}\left(\frac{1-T_k}{T_k}\right)-27\epsilon_0E_0^{*2}\left(\frac{1-T_k}{T_k}\right) \nonumber\\
& &-3 \left(\frac{\lambda_{k,-}^*T_k-\lambda_{k,-}^*}{T_k}\right)^2 -3\epsilon_0\left(\frac{1-T_k}{T_k}\right)^2E_0^{*2}<0,
\end{eqnarray}
where we use the fact $T_k-1<0$ for any integer $k>1$.
By continuity, as $\gamma\in[-\varepsilon,+\varepsilon]$ for $\varepsilon$ small enough, applying Theorem \ref{thm6.2} again, we obtain that
\begin{eqnarray}\label{eq3.24}
\lambda_{k,+}''(0)&=&\frac{\left\langle \widetilde{l},F_{\eta\eta\eta}(\lambda_{k,+}^*, 0)[\eta^*,\eta^*,\eta^*]\right\rangle }{-3\left\langle \widetilde{l}, F_{\lambda\eta}\left(\lambda_{k,+}^*,0\right)[1,\eta^*]\right\rangle}
<0
\end{eqnarray}
and
\begin{eqnarray}\label{eq3.25}
\lambda_{k,-}''(0)&=&\frac{\left\langle \widetilde{l},F_{\eta\eta\eta}(\lambda_{k,-}^*, 0)[\eta^*,\eta^*,\eta^*]\right\rangle }{-3\left\langle \widetilde{l}, F_{\lambda\eta}\left(\lambda_{k,-}^*,0\right)[1,\eta^*]\right\rangle}
>0
\end{eqnarray}
which indicates that the bifurcation at $\lambda_{k,+}^*$ is subcritical, however the bifurcation at $\lambda_{k,-}^*$ is supercritical.

In order to prove (ii) and (iii), we need to apply the exchange of stability Theorem \ref{thm6.3}. Based on the arguments above, we have known that $0$ is a simple eigenvalue of $T:=F_{\eta}\left(\lambda_{k,\pm}^*,0\right)$ and the transversal condition is satisfied.
Let the operator $K$ be the identical operator. Since $0$ is a simple eigenvalue of $T$, then $0$ is also a $K$-simple eigenvalue of $T$ (see the Definition \ref{def1} in Appendix). It is obvious that all of assumptions of Theorem \ref{thm6.3} hold.

Applying Theorem \ref{thm6.3} by letting $\psi^*=\cos(kx)$, we have that there are eigenvalues $\mu(s)$, $\beta(\lambda_{k,\pm})\in \mathbb{R}$ and their eigenfunctions $u(s)$,
$\psi(\lambda_{k,\pm})\in X $, such that
\begin{eqnarray}
F_{\eta}(\lambda_{k,\pm}(s),w(s))u(s)=\mu_{\pm}(s)u(s),\nonumber
\end{eqnarray}
\begin{eqnarray}
F_{\eta}(\lambda_{k,\pm},0)\psi(\lambda_{k,\pm})=\beta(\lambda_{k,\pm})\psi(\lambda_{k,\pm})\nonumber
\end{eqnarray}
with
\begin{align*}
\mu(0)=\beta\left(\lambda^*_{k,\pm}\right)=0,~~u(0)=\psi\left(\lambda^*_{k,\pm}\right)=\cos(kx)
\end{align*}
and
\begin{eqnarray}
\beta'\left(\lambda^*_{k,\pm}\right)\neq 0,~~\lim_{s\rightarrow 0,\beta(s)\neq 0}\frac{s\lambda_{k,\pm}'(s)}{\mu_{\pm}} =-\frac{1}{\beta'\left(\lambda^*_{k,\pm}\right)}.\nonumber
\end{eqnarray}

It follows from (\ref{eq3.14}) that
\begin{eqnarray}
\left[F_\eta(\lambda,0)\right]\cos(kx)&=&-\frac{2k}{\tanh(k)}\left(\lambda^2-\frac{\gamma}{k}\tanh(k)\lambda+\epsilon_0E_0^2-\frac{g+\sigma k^2}{k}\tanh(k)\right)\cos(kx).\nonumber
\end{eqnarray}
So we have that
\begin{eqnarray}
\beta'\left(\lambda^*_{k,\pm}\right)= \mp\frac{4k}{\tanh(k)}\sqrt{\frac{\gamma^2\tanh^2(k)}{4k^2}+\frac{(g+\sigma k^2)\tanh(k)}{k}-\epsilon_0E_0^2}.\nonumber
\end{eqnarray}
When $k>1$, we deduce that
\begin{eqnarray}
\beta'\left(\lambda^*_{k,+}\right)= -\frac{4k}{\tanh(k)}\sqrt{\frac{\gamma^2\tanh^2(k)}{4k^2}+\frac{(g+\sigma k^2)\tanh(k)}{k}-\epsilon_0E_0^2}<0 \nonumber
\end{eqnarray}
and
\begin{eqnarray}
\beta'\left(\lambda^*_{k,-}\right)= \frac{4k}{\tanh(k)}\sqrt{\frac{\gamma^2\tanh^2(k)}{4k^2}+\frac{(g+\sigma k^2)\tanh(k)}{k}-\epsilon_0E_0^2}>0. \nonumber
\end{eqnarray}
Since $\beta\left(\lambda^*_{k,\pm}\right)=0$, then we can deduce that $\beta\left(\lambda\right)<0$ if $\lambda>\lambda^*_{k,+}$ and $\beta\left(\lambda\right)>0$ if $\lambda<\lambda^*_{k,+}$ for $\left\vert \lambda-\lambda^*_{k,+}\right\vert$ sufficiently small. Similarly, $\beta\left(\lambda\right)<0$ if $\lambda>\lambda^*_{k,-}$ and $\beta\left(\lambda\right)>0$ if $\lambda<\lambda^*_{k,-}$ for $\left\vert \lambda-\lambda^*_{k,-}\right\vert$ sufficiently small. This implies near the bifurcation point $\lambda^*_{k,+}$ that the laminar solution is unstable for $\lambda<\lambda^*_{k,+}$ but formally stable for $\lambda>\lambda^*_{k,+}$. Similarly, near the bifurcation point $\lambda^*_{k,-}$, the laminar solution is formally stable for $\lambda<\lambda^*_{k,-}$ but unstable for $\lambda>\lambda^*_{k,-}$.

On the other hand, by using the Theorem \ref{thm6.3} again, we have that
\begin{eqnarray}\label{eq3.26}
\lim_{s\rightarrow 0,\beta(s)\neq 0}\frac{s\lambda_{k,\pm}'(s)}{\mu_{\pm}(s)} =\frac{1}{\pm\frac{4k}{\tanh(k)}\sqrt{\frac{\gamma^2\tanh^2(k)}{4k^2}+\frac{(g+\sigma k^2)\tanh(k)}{k}-\epsilon_0E_0^2}}.
\end{eqnarray}
It follows from (\ref{eq3.26}) that
\begin{eqnarray}\label{eq3.27}
\lim_{s\rightarrow 0,\beta(s)\neq 0}\frac{s\lambda_{k,+}'(s)}{\mu_{+}(s)} =\frac{1}{\frac{4k}{\tanh(k)}\sqrt{\frac{\gamma^2\tanh^2(k)}{4k^2}+\frac{(g+\sigma k^2)\tanh(k)}{k}-\epsilon_0E_0^2}},
\end{eqnarray}
\begin{eqnarray}\label{eq3.28}
\lim_{s\rightarrow 0,\beta(s)\neq 0}\frac{s\lambda_{k,-}'(s)}{\mu_{-}(s)} =\frac{1}{-\frac{4k}{\tanh(k)}\sqrt{\frac{\gamma^2\tanh^2(k)}{4k^2}+\frac{(g+\sigma k^2)\tanh(k)}{k}-\epsilon_0E_0^2}}.
\end{eqnarray}
It is known that $\lambda_{k,\pm}'(0)=0$. By using the Taylor expansion of $\lambda_{k,\pm}'(s)$ at $s=0$, we have that
\begin{eqnarray*}
\lambda_{k,\pm}'(s)=s \lambda_{k,\pm}''(0)+o\left(s\right).
\end{eqnarray*}
It follows from (\ref{eq3.27}) and (\ref{eq3.28}) that
\begin{eqnarray}
\lim_{s\rightarrow 0}\frac{s^2\lambda_{k,+}''(0)+o\left(s^2\right)}{\mu_{+}(s)}=\frac{1}{\frac{4k}{\tanh(k)}\sqrt{\frac{\gamma^2\tanh^2(k)}{4k^2}+\frac{(g+\sigma k^2)\tanh(k)}{k}-\epsilon_0E_0^2}}>0\nonumber
\end{eqnarray}
and
\begin{eqnarray}
\lim_{s\rightarrow 0}\frac{s^2\lambda_{k,-}''(0)+o\left(s^2\right)}{\mu_{-}(s)}=\frac{1}{-\frac{4k}{\tanh(k)}\sqrt{\frac{\gamma^2\tanh^2(k)}{4k^2}+\frac{(g+\sigma k^2)\tanh(k)}{k}-\epsilon_0E_0^2}}<0. \nonumber
\end{eqnarray}
Furthermore, by using (\ref{eq3.24}) and (\ref{eq3.25}), we obtain that
\begin{eqnarray}
\mu_{+}(s)<0,\qquad \mu_{-}(s)<0.\nonumber
\end{eqnarray}
for $\left|s\right|$ ($\neq0$) small enough. This implies that the nontrivial solutions are formally stable near the bifurcation points $\lambda_{k,\pm}^*$.
\end{proof}

\section{The existence of electrohydrodynamic ripples}
In section 3, the local bifurcation problem of (\ref{eq3.1}) was studied in the case when the kernel is one-dimensional. There we only choose $\lambda$ as the bifurcation parameter by fixing the electric field $E_0$. In this section, we will regard $E_0$ as another parameter and find that if $E_0^2$ is sufficiently close to the constant $E_{k,l}$ and the resonance phenomenon occurs by using the secondary bifurcation Theorem \ref{thm6.4} stated in the Appendix, that is to say there exist at least one secondary bifurcation branch emerging from some primary branch $\mathcal{K}_{k,\pm}$ at a point far from the set of $\{\eta=0\}$. These secondary bifurcation solutions are indeed corresponding to the electrohydrodynamic ripples.

To this end, let's first introduce a new functional analytic setting by choosing a Hilbert space because closed subspaces of Hilbert space would possess in general a closed complement. Then we redefine
$$
X:=\{ w:=\sum_{n=1}^{\infty}a_{n}\cos(nq) : a_{n}\in \mathbb{R},~~ \sum_{n=1}^{\infty}a_{n}^{2}n^{6}<\infty\}
$$
and
$$
Y:=\{v:= \sum_{n=0}^{\infty}a_{n}\cos(nq) : a_{n}\in \mathbb{R},~~ \sum_{n=1}^{\infty}a_{n}^{2}n^{2}<\infty \},
$$
which are identified as subspaces of $H^{3}(\mathbb{R}/(2\pi\mathbb{Z}))$ and $H^{1}(\mathbb{R}/(2\pi\mathbb{Z}))$, respectively. It is worth noting that $H^{3}(\mathbb{R}/(2\pi\mathbb{Z}))\hookrightarrow C^{2,\alpha}(\mathbb{R}/(2\pi\mathbb{Z}))$, which ensures that (\ref{eq3.1}) holds in the classic sense.
In addition, here we take $E_0$ as another parameter by considering
\begin{eqnarray} \label{eq4.1}
F(\lambda,E_0,\eta): & & =\psi_q^2+\frac{\eta'^2+1}{(1+\eta)^2}\psi_p^2-\frac{2\eta'}{1+\eta}\psi_q\psi_p+2g\eta-2\sigma\frac{\eta''}{\left(1+\eta'^2\right)^\frac{3}{2}}\nonumber\\
& &-\epsilon_0\left(V_q^2+\frac{\eta'^2+1}{(1-\eta)^2}V_p^2-\frac{2\eta'}{1-\eta}V_qV_p\right)\bigg|_{p=0}-Q=0
\end{eqnarray}
again, where $F: \mathbb{R}^2\times X\rightarrow Y$ and $V=V(E_0,\eta), \psi=\psi(\lambda,\eta)$. Define the set $\mathcal{O}:=\{\eta\in X : |w|<1\}$ and it is obvious that the operator $F$ defined by (\ref{eq4.1}) is real-analytic from $\mathbb{R}^2\times \mathcal{O}\subset \mathbb{R}^2\times X\rightarrow Y$.
Based on the arguments in the previous section, we have that

\begin{lemma} \label{lem4.1}
Let $T_{k}:=\frac{\tanh(k)}{k}$ for $k\in\mathbb{N}^+$ and assume that (\ref{eq3.16}) holds and
\begin{equation}\label{eq4.2}
\sigma\gamma^2\left(l^2-k^2\right)T_{k}T_{l}(T_{k}-T_{l})>\left( (g+\sigma k^2)T_k-(g+\sigma l^2)T_l \right)^2,
\end{equation}
then the Frechet derivative $F_{\eta}(\lambda,E_0,0): X\rightarrow Y $ is a Fourier multiplier. Furthermore, there holds that
\begin{equation}\label{eq4.3}
F_{\eta}(\lambda,E_0,0) \sum_{k=1}^{\infty}\eta_{k}\cos(kx)=\sum_{k=1}^{\infty}D_{k}(\lambda,E_0) \eta_{k}\cos(kx),
\end{equation}
where $D_{k}(\lambda,E_0)$ is defined in (\ref{eq3.14}) and the operator $F_{\eta}(\lambda,E_0,0): X\rightarrow Y$ is a Fredholm operator of index zero for any $(\lambda,E_0)\in \mathbb{R}^2$. More precisely, defining
\begin{equation}\label{eq4.4}
\lambda_{k,\pm}^{*}:=\frac{\gamma}{2} T_k\pm\sqrt{\frac{\gamma^2}{4}T_{k}^2+(g+\sigma k^2)T_{k}-\epsilon_0E_0^2}
\end{equation}
and
\begin{equation}\label{eq4.5}
E_{k,l}:=\frac{\sigma\gamma^2\left(l^2-k^2\right)T_{k}T_{l}(T_{k}-T_{l})-\left( (g+\sigma k^2)T_k-(g+\sigma l^2)T_l \right)^2}{(T_k-T_l)^2\gamma^2\epsilon_0}
\end{equation}
for any $k,l\in \mathbb{N}^+$, we have that

(1) $\lambda_{k,+}^*>0$, $\lambda_{k,-}^*<0$ and $E_{k,l}>0$;

(2) if $\lambda\notin \{\lambda_{k,\pm}^{*} : k\in \mathbb{N}^+\}$, then $F_{\eta}(\lambda,E_0,0)$ is an invertible operator;

(3) if $\lambda=\lambda_{k,\pm}^{*}$ and $E_0^2 =E_{k,l}$ for some positive integers $k\neq l$, then $\lambda_{k,+}^{*}=\lambda_{l,+}^{*}$ (or $\lambda_{k,-}^{*}=\lambda_{l,-}^{*}$). Moreover, $0$ is an eigenvalue of the operator $F_{\eta}(\lambda,E_0,0)$ and the corresponding eigenspace is two-dimensional.
\end{lemma}

\begin{proof}
We note that the sign of $\lambda_{k,\pm}^*$ follows from (\ref{eq3.16}) and the sign of the constant $E_{k,l}$ can be determined directly from (\ref{eq4.2}).

In order to prove \emph{(2)}, it follows from the arguments in previous section that zero is an eigenvalue of $F_{\eta}(\lambda,E_0,0)$ if and only if $D_{k}(\lambda,E_0)=0$ for any positive integer $k$. It is easy to check that $\lambda=\lambda_{k,\pm}^{*}$ are solutions to $D_{k}(\lambda,E_0)=0$, whereby $\lambda_{k,\pm}^{*}$ are given by (\ref{eq4.4}). Thus, if $\lambda\notin \{\lambda_{k,\pm}^{*} : k\in \mathbb{N}^+\}$, then $F_{\eta}(\lambda,E_0,0)$ is an invertible operator.

Now let's prove \emph{(3)} by letting
$$\lambda_{k,+}^*=\lambda_{l,+}^*~~ \text{or}~~\lambda_{k,-}^*=\lambda_{l,-}^*$$
for some $k\neq l$, we obtain by using algebraic manipulations that
$$E_0^2=E_{k,l}.$$
This means, as $E_0^2=E_{k,l}$, that the kernel space of $F_{\eta}(\lambda_{k,\pm}^*,E_0,0)$ is generated by $\{\cos(kx), \cos(lx)\}\in X$. In fact, we are left to show that there cannot be another integer $m\notin\{k,l\}$ such that $\lambda_{k,+}^*=\lambda_{l,+}^*=\lambda_{m,+}^*$ or $\lambda_{k,-}^*=\lambda_{l,-}^*=\lambda_{m,-}^*$ when $E_0^2=E_{k,l}$. By contradiction arguments, if $\lambda_{k,+}^*=\lambda_{l,+}^*=\lambda_{m,+}^*$ or $\lambda_{k,-}^*=\lambda_{l,-}^*=\lambda_{m,-}^*$ holds. Then we have that $E_0^2=E_{k,l}=E_{k,m}$, which leads to $T_{l}=T_{m}$ which is contradicted due to $m\neq l$.
\end{proof}

\begin{remark} \label{rem2}
By the property of the sequence $(T_k)_k$ of being decreasing, the assumption (\ref{eq4.2}) holds provided $\gamma^2$ is not small enough. Different from the case in \cite{MartinM}, the existence of ripples here is ensured not only by letting the electric field $E_0^2$ be close to $E_{k,l}$ but also by taking the suitable vorticity function $\gamma$.
\end{remark}

In the following, we will show that if (\ref{eq3.16}) and (\ref{eq4.2}) hold and the electric field $E_0^2$ is sufficiently close to the constant $E_{k,l}$ defined by (\ref{eq4.5}) with
$$\frac{k}{l}\in \mathbb{N}^+\setminus\{1\},$$
then there would be secondary bifurcation branches emerging from some of these primary branches obtained in Theorem \ref{thm3.1}. Now we further make a decomposition of Hilbert spaces $X$ and $Y$ by setting
$$
X_{1}:=\{ w\in X: w=\sum_{n=1}^{\infty}a_{n}\cos(knq) \},
$$
$$
X_{2}:=\{ w\in X: w=\sum_{\frac{n}{k}\notin \mathbb{N}}a_{n}\cos(nq) \}
$$
and
$$
Y_{1}:=\{v\in Y: v=\sum_{n=1}^{\infty}a_{n}\cos(knq)\}, ~~~Y_{2}:=\{v\in Y: v=\sum_{\frac{n}{k}\notin \mathbb{N}}a_{n}\cos(nq)\}.
$$

Once the electric field $E_0^*$ is chosen such that $E_0^{*2}=E_{k,l}$ for $k,l\in \mathbb{N}^+$, then it follows from Lemma \ref{lem4.1} that either $\lambda_{k,+}^{*}=\lambda_{l,+}^{*}$ or $\lambda_{k,-}^{*}=\lambda_{l,-}^{*}$.
For convenience, we will denote the bifurcation points by
$$
\lambda_{\pm}^{*}=\lambda_{k,\pm}^{*}=\lambda_{l,\pm}^{*}
$$
Then the operator $F_{\eta}(\lambda_{\pm}^{*},E_0^*,0)$ has a two-dimensional kernel. Moreover,
$$\mathcal{N}(F_{\eta}(\lambda_{\pm}^{*},E_0^*,0))=span\{ x_1,x_2 \},$$
where $x_1:=\cos(kq)$ and $x_2:=\cos(lq)$.
Now let us state the main theorem of this section.

\begin{theorem} \label{thm4.1} (The existence of secondary bifurcation branches)
Assume that (\ref{eq3.16}) and (\ref{eq4.2}) hold and $|E_0-E_0^*|<\delta$ for small $\delta>0$. Then there exists a smooth curve $\mathcal{S}_{E_0}$ that intersects $\mathcal{K}_{k,\pm}$ and the secondary bifurcation curve $\mathcal{S}_{E_0}$ consists only of solutions $(\lambda, \eta)$ of problem (\ref{eq4.1}) with minimal period $\frac{2\pi}{l}$.
\end{theorem}
\begin{figure}[ht] \label{fig4}
\centering
\includegraphics[width=0.75\textwidth]{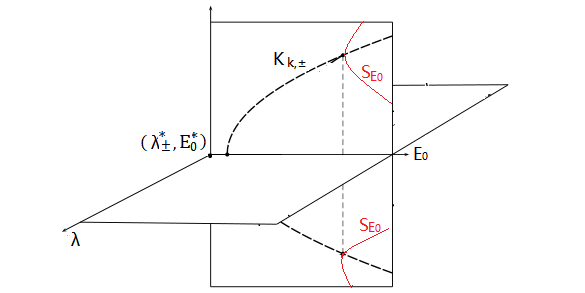}
\centering
\caption{The secondary bifurcation curve $\mathcal{S}_{E_0}$.}
\end{figure}

\begin{proof}
To finish the proof, we need to verify the assumptions (H1)-(H8) in Theorem \ref{thm6.2}. Let us first make a constraint for spaces by
$$
F: \mathbb{R}^2\times X_1\rightarrow Y_1.
$$
Then it follows from (\ref{eq3.17}) that $F(\lambda, E_0, 0)=0$ for any $(\lambda, E_0)\in \mathbb{R}^2$. The Fredholm property of operator $F_{\eta}(\lambda_{\pm}^*,E_0^*,0)$ from $X_1$ to $Y_1$ can been deduced similarly as in Subsection 3.2. In addition, it follows from (\ref{eq3.18}) that
\begin{equation}\label{eq4.6}
F_{\lambda\eta}(\lambda_{\pm}^*,E_0^*,0)[1,x_1]=-\frac{2k}{\tanh(k)}\left(2\lambda_{\pm}^*-\frac{\gamma}{k}\tanh(k)\right)x_1\notin \mathcal{R}(F_{\eta}(\lambda_{\pm}^*,E_0^*,0)).
\end{equation}
It is obvious that $x_{1}=\psi_1=\cos(kq)$ and $x_{2}=\psi_2=\cos(lq)$ in our setting. The fact $x_i\in X_i$ and $\psi_i\in Y_i$ for $i=1,2$ together with the Lemma \ref{lem4.1} ensures that (H1)-(H6) hold naturally. Similarly as (\ref{eq4.6}), we have that
\begin{equation}\label{eq4.7}
F_{\lambda\eta}(\lambda_{\pm}^*,E_0^*,0)[1,x_2]=-\frac{2l}{\tanh(l)}\left(2\lambda_{\pm}^*-\frac{\gamma}{l}\tanh(l)\right)x_2\notin \mathcal{R}(F_{\eta}(\lambda_{\pm}^*,E_0^*,0)),
\end{equation}
which verifies the general transversality condition (H7).

Finally, we are left to prove the nondegeneracy condition (H8) holds. Based on (\ref{eq4.3}), we have that
\begin{eqnarray}\label{eq4.8}
F_{E_0\eta}(\lambda_{\pm}^*,E_0^*, 0)[1,x_1]=-\frac{4k}{\tanh(k)} \epsilon_0E_0^*x_1, \nonumber \\
F_{E_0\eta}(\lambda_{\pm}^*,E_0^*, 0)[1,x_2]=-\frac{4l}{\tanh(l)} \epsilon_0E_0^*x_2.
\end{eqnarray}
With (\ref{eq4.6})-(\ref{eq4.8}) in hands, we obtain that
\begin{eqnarray}  \label{eq4.9}
~& &\left|
\begin{array}{cc}
\langle F_{\lambda \eta}(\lambda_{\pm}^*,E_0^*, 0)[1,\cos(kq)]|\cos(kq) \rangle  ~&~  \langle F_{E_0 \eta}(\lambda_{\pm}^*,E_0^*, 0)[1,\cos(kq)]|\cos(kq) \rangle \\
~ & ~\\
\langle F_{\lambda \eta}(\lambda_{\pm}^*,E_0^*, 0)[1,\cos(lq)]|\cos(lq) \rangle   ~&~ \langle F_{E_0 \eta}(\lambda_{\pm}^*,E_0^*, 0)[1,\cos(lq)]|\cos(lq) \rangle  \\
\end{array}
\right|  \nonumber\\
~\nonumber\\
& &=\frac{8\gamma\epsilon_0E_0^*(T_l-T_k)}{T_kT_l}\neq 0,
\end{eqnarray}
where we use the facts of $k\neq l$ and $\{T_k\}_{k}$ is a decreasing sequence.
On the other hand, it follows from (\ref{eq3.12})-(\ref{eq3.13}) that
\begin{eqnarray} \label{eq4.10}
w_{k,p}(0)=\frac{(\gamma+\lambda)T_k-\lambda}{T_k}  ,\quad h_{k,p}(0)=\left(\frac{1-T_k}{T_k}\right)E_0.
\end{eqnarray}
Combining (\ref{eq4.10}) with (\ref{eq5.4}) (see Subsection 5.2), we have that
\begin{eqnarray} \label{eq4.11}
~& &F_{\eta\eta}(\lambda_{\pm}^*,E_0^*, 0)[\cos(kq),\cos(kq)]
~\nonumber\\
& &=2(\lambda_{\pm}^{*2}-\epsilon_0E_0^{*2}) k^2\sin^2(kq)+6(\lambda_{\pm}^{*2}-\epsilon_0E_0^{*2})\cos^2(kq)     \nonumber\\
& &-2\epsilon_0 \frac{(1-T_k)^2}{T_k^2}E_0^{*2}\cos^2(kq)+2 \frac{\left((\gamma+\lambda_{\pm}^*)T_k-\lambda_{\pm}^*\right)^2}{T_k^2}\cos^2(kq)  \nonumber\\
& & -8\epsilon_0E_0^{*2}\frac{1-T_k}{T_k}\cos^2(kq)     -8 \lambda_{\pm}^*\frac{(\gamma+\lambda_{\pm}^*)T_k-\lambda_{\pm}^*}{T_k}\cos^2(kq)
\end{eqnarray}
and
\begin{eqnarray} \label{eq4.12}
~& &F_{\eta\eta}(\lambda_{\pm}^*,E_0^*, 0)[\cos(kq),\cos(lq)]
~\nonumber\\
& &=2(\lambda_{\pm}^{*2}-\epsilon_0E_0^{*2}) kl\sin(kq)\sin(lq)+6(\lambda_{\pm}^{*2}-\epsilon_0E_0^{*2})\cos(kq)\cos(lq)   \nonumber\\
& &-2\epsilon_0 \frac{(1-T_k)(1-T_l)}{T_kT_l}E_0^{*2}\cos(kq) \cos(lq) \nonumber\\
& & +2\frac{\left((\gamma+\lambda_{\pm}^*)T_k-\lambda_{\pm}^*\right)\left((\gamma+\lambda_{\pm}^*)T_l-\lambda_{\pm}^*\right)}{T_kT_l}\cos(kq)\cos(lq) \nonumber\\
& &-8\epsilon_0E_0^{*2}\frac{1-T_k}{T_k}\cos(kq) \cos(lq)-8 \lambda_{\pm}^*\frac{(\gamma+\lambda_{\pm}^*)T_k-\lambda_{\pm}^*}{T_k}\cos(kq) \cos(lq). \nonumber\\
& &~
\end{eqnarray}
Therefore, it follows from (\ref{eq4.6})-(\ref{eq4.7}) and (\ref{eq4.11})-(\ref{eq4.12}) that
\begin{eqnarray}\label{eq4.13}
~& &\left|
\begin{array}{cc}
\langle F_{\lambda \eta}(\lambda_{\pm}^*,E_0^*, 0)[1,\cos(kq)]|\cos(kq) \rangle   ~&~  \langle F_{\eta\eta}(\lambda_{\pm}^*,E_0^*, 0)[\cos(kx),\cos(kq)]|\cos(kq)  \rangle \\
~ & ~\\
\langle F_{\lambda \eta}(\lambda_{\pm}^*,E_0^*, 0)[1,\cos(lq)]|\cos(lq)  \rangle   ~&~ \langle F_{\eta\eta}(\lambda_{\pm}^*,E_0^*, 0)[\cos(kx),\cos(lq)]|\cos(lq) \rangle  \\
\end{array}
\right| \nonumber\\
~ \nonumber\\
& &=-\frac{2k}{\tanh(k)}\left(2\lambda_{\pm}^*-\frac{\gamma}{k}\tanh(k)\right)
\frac{M}{\pi}\int^{\pi}_{-\pi} \cos(kq) \cos^2(lq)dx\neq0,
\end{eqnarray}
where we use the following facts for $k\neq l$ that
$$
\int^{\pi}_{-\pi} \cos^3(kq)dq=0, \quad \int^{\pi}_{-\pi} \sin^2(kq) \cos(kq)dq=0,\quad \int^{\pi}_{-\pi} \cos(kq) \cos^2(lq)dq\neq 0
$$
and
\begin{eqnarray} \label{eq4.14}
& &M=6(\lambda_{\pm}^{*2}-\epsilon_0E_0^{*2})-2\epsilon_0 \frac{(1-T_k)(1-T_l)}{T_kT_l}E_0^{*2} \nonumber\\
& &+2\frac{\left((\gamma+\lambda_{\pm}^*)T_k-\lambda_{\pm}^*\right)\left((\gamma+\lambda_{\pm}^*)T_l-\lambda_{\pm}^*\right)}{T_kT_l}
-8\epsilon_0E_0^{*2}\frac{1-T_k}{T_k} \nonumber\\
& &-8 \lambda_{\pm}^*\frac{(\gamma+\lambda_{\pm}^*)T_k-\lambda_{\pm}^*}{T_k}.
\end{eqnarray}
Indeed, we can check that $M\neq0$ in (\ref{eq4.14}) by using (\ref{eq4.4}) and (\ref{eq4.5}) for some integers $k\neq l$. Up to now, the assumptions (H1)-(H8) are verified, thus we finish the proof by using the Theorem \ref{thm6.4} in Appendix.
\end{proof}

\begin{remark}
The local stability of primary solution curves $\mathcal{K}_{k,\pm}$ has been given in Subsection 3.3. However, as far as we know that the local stability of secondary bifurcation curve $\mathcal{S}_{E_0}$ is still open. To this end, maybe some new exchange of stability theorems are needed to establish, which is a very interesting project in future.
\end{remark}

\section{Proofs and calculations}
We first give the Green's function $G_2(x,t)$ of the operator $\frac{\partial}{\partial x^2}-k^2$ on the interval $[0,1]$ with Dirichlet boundary conditions and the Green's function $G_1(x,t)$ of the operator $\frac{\partial}{\partial x^2}-k^2$ on the interval $[-1,0]$ with Dirichlet boundary conditions, which are used in Section 3. In addition, we show the second derivative of the abstract operator $F(\lambda, E_0,\eta)$ with responding to $\eta$, which plays a key role in checking the nondegeneracy condition in Section 4. Moreover, we also give the third derivative of the abstract operator $F(\lambda, E_0,\eta)$ with responding to $\eta$, which is used to judge the direction of bifurcation.
\subsection{\bf The Green's functions}

To get the Green's function $G_2(x,t)$, let us first consider the following ODE
\begin{eqnarray}
\left\{\begin{array}{llll}{\left(\frac{\partial}{\partial x^2}-k^2\right)G_1(x,t)=\delta(x-t)}, \\
{G_2(0,t)=0,\quad G_2(1,t)=0} ,\end{array}\right. \nonumber
\end{eqnarray}
where $\delta$ is the Dirac function. For $x<t$, by using the boundary condition $G_1(0,t)=0$, we have that
$$
G_2(x,t)=A(t)\left(e^{kx}-e^{-kx}\right).
$$
For $x>t$, by using the boundary condition $G_1(1,t)=0$, we have that
$$
G_2(x,t)=B(t)\left(e^{kx}-e^{k(2-x)}\right).
$$
Then, for $x=t$, we have that
$$
A(t)\left(e^{kt}-e^{-kt}\right)=B(t)\left(e^{kt}-e^{k(2-t)}\right)
$$
and
$$
B(t)\frac{\partial\left(e^{kx}-e^{k(2-x)}\right) }{\partial x}\bigg|_{t}-A(t)\frac{\partial\left(e^{kx}-e^{-kx}\right) }{\partial x}\bigg|_{t}=1,
$$
which gives that
$$
A(t)=\frac{e^{kt}-e^{k(2-t)}}{2k(e^{2k}-1)}, \quad B(t)=\frac{e^{kt}-e^{-kt}}{2k(e^{2k}-1)}=\frac{\sinh(kt)}{k(e^{2k}-1)}.
$$
Taking $A(t)$ and $B(t)$ into the formulas above, we have that
\begin{eqnarray}
G_2(x,t)=-\frac{1}{k\sinh(k)}\left\{\begin{array}{llll}{\sinh(kx)\sinh(k(1-t))} & {\text { for }~x\leq t}, \\
~\\
{\sinh(k(1-x))\sinh(kt)} & { \text { for }~x\geq t}.\end{array}\right. \nonumber
\end{eqnarray}

Similarly, by considering the ODE
\begin{eqnarray}
\left\{\begin{array}{llll}{\left(\frac{\partial}{\partial x^2}-k^2\right)G_1(x,t)=\delta(x-t)}, \\
{G_1(-1,t)=0,\quad G_1(0,t)=0} ,\end{array}\right. \nonumber
\end{eqnarray}
we can obtain that
\begin{eqnarray}
G_1(p,r)=\frac{1}{k\sinh(k)}\left\{\begin{array}{llll}{\sinh(k(1+x))\sinh(kt)} & {\text { for }~x\leq t}, \\
~\\
{\sinh(kx))\sinh(k(1+t))} & { \text { for }~x\geq t}.\end{array}\right. \nonumber
\end{eqnarray}

\subsection{\bf The second and third derivative}
It is known by letting $d=1$ that
\begin{eqnarray}  \label{eq5.1}
F(\lambda,E_0,\eta): & & =\psi_q^2+\frac{\eta'^2+1}{(1+\eta)^2}\psi_p^2-\frac{2\eta'}{1+\eta}\psi_q\psi_p+2g\eta-2\sigma\frac{\eta''}{\left(1+\eta'^2\right)^\frac{3}{2}}\nonumber\\
& &-\epsilon_0\left(V_q^2+\frac{\eta'^2+1}{(1-\eta)^2}V_p^2-\frac{2\eta'}{1-\eta}V_qV_p\right)\bigg|_{p=0}-Q=0, \nonumber\\
\end{eqnarray}
where $F: \mathbb{R}^2\times X\rightarrow Y$, $V=V(E_0,\eta), \psi=\psi(\lambda,\eta)$
and $F$ is analytic in the open set $\mathcal{O}$. Therefore, we have that
\begin{eqnarray} \label{eq5.2}
~& &F_{\eta\eta}(\lambda_{\pm}^*,E_0^*, 0)[\eta,\eta]
~\nonumber\\
& &=\left(F_{\eta\eta}[\eta,\eta]+F_{\eta V}[\eta,\frac{\partial V}{\partial\eta}[\eta]]+
F_{\eta\psi}[\eta,\frac{\partial\psi}{\partial\eta}[\eta]\right)\bigg|_{(\lambda_{\pm}^*,E_0^*, 0),p=0}  \nonumber\\
& &+\left(F_{V\eta}[\frac{\partial V}{\partial\eta}[\eta],\eta]+F_{VV}[\frac{\partial V}{\partial\eta}[\eta],\frac{\partial V}{\partial\eta}[\eta]]+
F_{V\psi}[\frac{\partial V}{\partial\eta}[\eta],\frac{\partial \psi}{\partial\eta}[\eta]]\right)\bigg|_{(\lambda_{\pm}^*,E_0^*, 0),p=0}  \nonumber\\
& &+\left(F_{\psi\eta}[\frac{\partial \psi}{\partial\eta}[\eta],\eta]+F_{\psi V}[\frac{\partial \psi}{\partial\eta}[\eta],\frac{\partial V}{\partial\eta}[\eta]]+
F_{\psi\psi}[\frac{\partial \psi}{\partial\eta}[\eta],\frac{\partial \psi}{\partial\eta}[\eta]]\right)\bigg|_{(\lambda_{\pm}^*,E_0^*, 0),p=0}  \nonumber\\
& &=\left(F_{\eta\eta}[\eta,\eta]+F_{VV}[\frac{\partial V}{\partial\eta}[\eta],\frac{\partial V}{\partial\eta}[\eta]]+F_{\psi\psi}[\frac{\partial \psi}{\partial\eta}[\eta],\frac{\partial \psi}{\partial\eta}[\eta]]\right)\bigg|_{(\lambda_{\pm}^*,E_0^*, 0),p=0}  \nonumber\\
& &+2\left(F_{\eta V}[\eta,\frac{\partial V}{\partial\eta}[\eta]]+
F_{\eta\psi}[\eta,\frac{\partial \psi}{\partial\eta}[\eta]]\right)\bigg|_{(\lambda_{\pm}^*,E_0^*, 0),p=0}  \nonumber\\
& &=2\left(\overline{\psi}_{p}^2\left(\eta'^{2}+3\eta^2\right)-\epsilon_{0}\overline{V}_{p}^2\left(\eta'^{2}+3\eta^2\right)
-\epsilon_{0}\left(h_q^2+h_p^2\right)+w_q^2+w_p^2\right)\bigg|_{(\lambda_{\pm}^*,E_0^*, 0),p=0}  \nonumber\\
& &+4\left(-2\epsilon_0\overline{V}_{p}\eta h_{p}+\epsilon_0\overline{V}_{p}\eta'h_q-2\overline{\psi}_{p}\eta w_{p}-\overline{\psi}_{p}\eta' w_{q}\right)\bigg|_{(\lambda_{\pm}^*,E_0^*, 0),p=0}
\end{eqnarray}
It follows from (\ref{eq2.5}) that
$$
\overline{V}=E_0 y=E_0(p(1-\eta)+\eta )
$$
and
 $$
\overline{\psi}=\frac{\gamma y^2}{2}+\lambda y=\frac{\gamma}{2}\left(p(1+\eta)+\eta\right)^2+\lambda\left(p(1+\eta)+\eta\right).
$$
Thus, the formula (\ref{eq5.2}) would become
\begin{eqnarray} \label{eq5.3}
~& &F_{\eta\eta}(\lambda_{\pm}^*,E_0^*, 0)[\eta,\eta]
~\nonumber\\
& &=2\left(\lambda_{\pm}^{*2}-\epsilon_{0}E_0^{*2}\right)\left(\eta'^{2}+3\eta^2\right)
-2\epsilon_{0}\left(h_q^2(q,0)+h_p^2(q,0)\right)\bigg|_{(\lambda_{\pm}^*,E_0^*, 0)}  \nonumber\\
& &+2w_q^2(q,0)+2w_p^2(q,0)-8\epsilon_0E_0^*\eta h_{p}(q,0)+4\epsilon_0E_0^*\eta'h_q(q,0)\bigg|_{(\lambda_{\pm}^*,E_0^*, 0)}  \nonumber\\
& &-8\lambda_{\pm}^*\eta w_{p}(q,0)-4\lambda_{\pm}^*\eta' w_{q}(q,0)\bigg|_{(\lambda_{\pm}^*,E_0^*, 0)}.
\end{eqnarray}
In addition, from (\ref{eq3.10}) and (\ref{eq3.11}), we have that
\begin{eqnarray}
w_{k}(0)=0  ,\quad h_{k}(0)=0,   \nonumber
\end{eqnarray}
which implies that
\begin{eqnarray} \label{eq5.4}
~& &F_{\eta\eta}(\lambda_{\pm}^*,E_0^*, 0)[\eta,\eta]
~\nonumber\\
& &=2\left(\lambda_{\pm}^{*2}-\epsilon_{0}E_0^{*2}\right)\left(\eta'^{2}+3\eta^2\right)
-2\epsilon_{0}\left(h_q^2(q,0)+h_p^2(q,0)\right)\bigg|_{(\lambda_{\pm}^*,E_0^*, 0)}  \nonumber\\
& &+2w_q^2(q,0)+2w_p^2(q,0)-8\epsilon_0E_0^*\eta h_{p}(q,0)+4\epsilon_0E_0^*\eta'h_q(q,0)\bigg|_{(\lambda_{\pm}^*,E_0^*, 0)}  \nonumber\\
& &-8\lambda_{\pm}^*\eta w_{p}(q,0)-4\lambda_{\pm}^*\eta' w_{q}(q,0)\bigg|_{(\lambda_{\pm}^*,E_0^*, 0)}\nonumber\\
& &=2\left(\lambda_{\pm}^{*2}-\epsilon_{0}E_0^{*2}\right)\left(\eta'^{2}+3\eta^2\right)
-2\epsilon_{0}h_p^2(q,0)\bigg|_{(\lambda_{\pm}^*,E_0^*, 0)}  \nonumber\\
& &+2w_p^2(q,0)-8\epsilon_0E_0^*\eta h_{p}(q,0)-8\lambda_{\pm}^*\eta w_{p}(q,0)\bigg|_{(\lambda_{\pm}^*,E_0^*, 0)}.
\end{eqnarray}

In addition, it is obvious that
\begin{eqnarray} \label{eq5.5}
~& &F_{\eta\eta\eta}(\lambda_{\pm}^*,E_0^*, 0)[\eta,\eta,\eta]
~\nonumber\\
& &=\left(F_{\eta\eta\eta}[\eta,\eta,\eta]+F_{VVV}[\frac{\partial V}{\partial\eta}[\eta],\frac{\partial V}{\partial\eta}[\eta],\frac{\partial V}{\partial\eta}[\eta]]+F_{\psi\psi\psi}[\frac{\partial \psi}{\partial\eta}[\eta],\frac{\partial \psi}{\partial\eta}[\eta]],\frac{\partial \psi}{\partial\eta}[\eta]\right)\bigg|_{(\lambda_{\pm}^*,E_0^*, 0),p=0}  \nonumber\\
& &+3\left(F_{\eta\eta V}[\eta,\eta,\frac{\partial V}{\partial\eta}[\eta]]+
F_{\eta\eta\psi}[\eta,\eta,\frac{\partial \psi}{\partial\eta}[\eta]]\right)\bigg|_{(\lambda_{\pm}^*,E_0^*, 0),p=0}  \nonumber\\
& &+3\left(F_{\eta V V}[\eta,\frac{\partial V}{\partial\eta}[\eta],\frac{\partial V}{\partial\eta}[\eta]]+
F_{\eta\psi\psi}[\eta,\frac{\partial \psi}{\partial\eta}[\eta],\frac{\partial \psi}{\partial\eta}[\eta]]\right)\bigg|_{(\lambda_{\pm}^*,E_0^*, 0),p=0}
\end{eqnarray}
Similarly, considering the $\overline{V}$, $\overline{\psi}$ and $w_{k}(0)=0, h_{k}(0)=0$, then we have that
\begin{eqnarray} \label{eq5.6}
~& &F_{\eta\eta\eta}(\lambda_{\pm}^*,E_0^*, 0)[\eta,\eta,\eta]
~\nonumber\\
& &=\left(\overline{\psi}_{p}^2\left(-12\eta'^{2}\eta-24\eta^3\right)-\epsilon_{0}\overline{V}_{p}^2\left(12\eta'^{2}\eta+24\eta^3\right)
+18\sigma\eta''\eta'^2\right)\bigg|_{(\lambda_{\pm}^*,E_0^*, 0),p=0}  \nonumber\\
& &+12\left(\eta'^2+3\eta^2\right)\overline{\psi}_{p}w_p-12\epsilon_0\left(\eta'^2+3\eta^2\right)\overline{V}_{p}h_p-4\eta w_p^2-4\epsilon_0\eta h_p^2\bigg|_{(\lambda_{\pm}^*,E_0^*, 0),p=0}   \nonumber\\
& &=\left(-\lambda_{\pm}^{*2}\left(12\eta'^{2}\eta+24\eta^3\right)-\epsilon_{0}E_0^{*2}\left(12\eta'^{2}\eta+24\eta^3\right)
+18\sigma\eta''\eta'^2\right)\bigg|_{(\lambda_{\pm}^*,E_0^*, 0),p=0}  \nonumber\\
& &+12\left(\eta'^2+3\eta^2\right)\lambda_{\pm}^*w_p-12\epsilon_0\left(\eta'^2+3\eta^2\right)E_0^*h_p-4\eta w_p^2-4\epsilon_0\eta h_p^2\bigg|_{(\lambda_{\pm}^*,E_0^*, 0),p=0}  \nonumber\\
& &=-12\left(\lambda_{\pm}^{*2}+ \epsilon_{0}E_0^{*2} \right)\eta'^{2}\eta-24\left(\lambda_{\pm}^{*2}+ \epsilon_{0}E_0^{*2} \right)\eta^3
+18\sigma\eta''\eta'^2\bigg|_{(\lambda_{\pm}^*,E_0^*, 0),p=0}  \nonumber\\
& &+12\lambda_{\pm}^*\eta'^2w_p+36\lambda_{\pm}^*\eta^2w_p-12\epsilon_0E_0^*\eta'^2h_p-36\epsilon_0E_0^*\eta^2h_p\bigg|_{(\lambda_{\pm}^*,E_0^*, 0),p=0} \nonumber\\
& &-4\eta w_p^2-4\epsilon_0\eta h_p^2\bigg|_{(\lambda_{\pm}^*,E_0^*, 0),p=0}.
\end{eqnarray}

\section*{Appendix: Quoted results}
This appendix not only collects the remarkable Crandall-Rabinowitz local bifurcation theorem \cite{CrandallR}, which are designed to deal with the case of linearized operator with one dimensional kernel, but also gives the secondary bifurcation theorem due to Shearer \cite{Shearer}. In addition, we give the definition of stability and recall the stability exchange theorem due to Crandall and Rabinowitz \cite{CrandallR1} and some useful formulas used to judge the direction of local bifurcation.

\begin{theorem}(Local bifurcation theorem \cite{CrandallR}) \label{thm6.1}
Let $X$ and $Y$ be Banach spaces and $F: \mathbb{R}\times X \rightarrow Y$ be a $C^k$ function with $k\geq 2$. Let $\mathcal{N}(T)$ and $\mathcal{R}(T)$ denote the null space and range of any operator $T$. Suppose that

(H1) $F(\lambda,0)=0$ for all $\lambda\in \mathbb{R}$;

(H2) For some $\lambda_*\in \mathbb{R}$, $ F_{x}\left(\lambda_*, 0\right)$ is a Fredholm operator with $\mathcal{N} \left( F_{x}\left(\lambda_*, 0\right)\right)$ and $Y /\mathcal{R}\left( F_{x}\left(\lambda_*, 0\right)\right)$ are $1$-dimensional and
the null space generated by $x_*$, and the transversality condition
\begin{equation}
F_{\lambda,x}\left(\lambda_*, 0\right)[1,x_*]\not\in\mathcal{R}\left( F_{x}\left(\lambda_*, 0\right)\right)\nonumber
\end{equation}
holds, where $\mathcal{N} \left( F_{x}\left(\lambda_*, 0\right)\right)$ and $\mathcal{R}\left( F_{x}\left(\lambda_*, 0\right)\right)$ denote null space and range space of $ F_{x}\left(\lambda_*, 0\right)$, respectively.

Then $\lambda_*$ is a bifurcation point in the sense that there exists $\varepsilon>0$ and a primary branch $\mathcal{P}$ of solutions
$$
\{ (\lambda,x)=(\Lambda(s),s\psi(s)):s\in\mathbb{R}, |s|<\varepsilon \}\subset\mathbb{R}\times X
$$
with $F(\lambda,x)=0$, $\Lambda(0)=0$, $\psi(0)=x_*$ and the maps
$$
s\mapsto\Lambda(s)\in\mathbb{R},~~~~s\mapsto s\psi(s)\in X
$$
are of class $C^{k-1}$ on $(-\varepsilon, \varepsilon)$.
\end{theorem}

\begin{theorem}(The formula of bifurcation direction \cite{Kielhofer}) \label{thm6.2}
If $F$ satisfies the hypotheses of Theorem \ref{thm6.1}, then
\begin{equation}
\lambda'(0)=-\frac{\left\langle l,F_x^{(2)}\left(\lambda_*,0\right)x_*^2\right\rangle}{2\left\langle l, F_{\lambda x}(\lambda_*,0)x_*\right\rangle},\nonumber
\end{equation}
where $l\in X^*$ satisfying $\mathcal{N}(l)=\mathcal{R}\left(D_x F(\lambda_*,0)\right)$, and $X^*$ being the dual space of $X$.
Furthermore, for $n\geq2$, if $F_x^{(j)}(\lambda_*,0)x_*^j=0$ for $j\in\{1,\ldots,n\}$, then $\lambda^{(j)}(0)=0$, $\psi^{(j)}(0)=0$ for $j\in\{1,\ldots,n-1\}$ and
\begin{equation}
\lambda^{(n)}(0)=-\frac{\left\langle l,F_x^{(n+1)}\left(\lambda_*,0\right)x_*^{n+1}\right\rangle}{(n+1)\left\langle l, F_{\lambda x}(\lambda_*,0)x_*\right\rangle},\nonumber
\end{equation}
where $F_x^{(j)}(\lambda_*,0)x_*^j$ means the value of the $j$-th Fr\'{e}chet derivative of $F(\lambda_*,x)$ with respect to $x$ at $(\lambda_*,0)$
evaluated at the $j$-tuple each of whose entries is $x_*$.
\end{theorem}

Now let us recall the concept of stability and the definition of $K-$simple eigenvalue introduced in \cite{CrandallR1}. The operator equation $F(\lambda,x)=0$ can be regarded as the equilibrium form of the evolution equation
\begin{eqnarray}
\frac{dx}{dt}=F(\lambda,x). \nonumber
\end{eqnarray}
Assume that $F\left(\lambda_0, x_0\right)=0$. If all of eigenvalues of $F_x\left(\lambda_0, x_0\right)$ are negative, then the solution $x_0$ is called asymptotically linear stable solution of $\frac{dx}{dt}=F(\lambda,x)$. If there exists a positive eigenvalue of $F_x\left(\lambda_0, x_0\right)$, $x_0$ is called unstable. If there exists a zero eigenvalue of $F_x\left(\lambda_0, x_0\right)$, $x_0$ is called neutral stable.

\begin{definition}\label{def1}

Let $T,K$: $X\rightarrow Y$ be two bounded linear operators from a real Banach space $X$ to another one $Y$. A complex number $\beta$ is called a $K-$simple eigenvalue of $T$ if
$$
dim \mathcal{N}(T-\beta K)=1=codim \mathcal{R}(T-\beta K)
$$
and
$$
K \psi^* \notin \mathcal{R}(T-\beta K) ~~for ~~0\neq \psi^*\in \mathcal{N}(T-\beta K).
$$
\end{definition}

If $K$ is the identity operator, then $K-$simple eigenvalue is called simple.
Next, let us show the Crandall-Rabinowitz exchange of stability theorem \cite{CrandallR1}.

\begin{theorem} \label{thm6.3}

Let $X$ and $Y$ be real Banach spaces and let $K, T: X\rightarrow Y$ be two bounded linear operators. Assume $F: \mathbb{R}\times X\rightarrow Y$ is $C^{2}$ near $\left(\lambda_*,0\right)\in \mathbb{R}\times X$ with $F(\lambda,0)=0$ for $\left\vert\lambda_*-\lambda\right\vert$ sufficiently small. Let $T=F_x\left(\lambda_*,0\right)$. If $\beta=0$ is a $F_{\lambda x}\left(\lambda_*,0\right)-$simple eigenvalue of operator $T$ and a $K-$simple eigenvalue of $T$, then there exists locally a curve $(\lambda(s),x(s))\in \mathbb{R}\times X$ such that
$$
(\lambda(0),x(0))=\left(\lambda_*,0\right)~~and ~~F(\lambda(s),x(s))=0.
$$
Moreover, if $F(\lambda,x)=0$ with $x\neq 0$ and $(\lambda,x)$ near $\left(\lambda_*,0\right)$, then
$$
(\lambda,x)=(\lambda(s),x(s))~~for~~some~~s\neq 0.
$$
Furthermore, there are eigenvalues $\beta(s)$, $\beta_{triv}(\lambda)\in \mathbb{R}$ with eigenvectors $\psi(s)$, $\psi_{triv}(\lambda)\in X $, such that
\begin{eqnarray}
F_{x}\left(\lambda(s),x(s)\right)\psi(s)=\beta(s)K\psi(s),\nonumber
\end{eqnarray}
\begin{eqnarray}
F_{x}\left(\lambda,0\right)\psi_{triv}(\lambda)=\beta_{triv}(\lambda)K\psi_{triv}(\lambda)\nonumber
\end{eqnarray}
with
\begin{align*}
\beta(0)=\beta_{triv}\left(\lambda_*\right)=0,~~\psi(0)=\psi_{triv}\left(\lambda_*\right)=\psi^*.
\end{align*}
Each curve is $C^{1}$ if $F$ is $C^{2}$, then
\begin{eqnarray}
\frac{d\beta_{triv}\left(\lambda\right)}{d\lambda}|_{\lambda=\lambda_*}\neq 0,~~\lim_{s\rightarrow 0,\beta\left(s\right)\neq 0}\frac{s\lambda'(s)}{\beta(s)}=-\frac{1}{\beta_{triv}'\left(\lambda_*\right)}.\nonumber
\end{eqnarray}
\end{theorem}

It's worth noting that $\frac{d\beta_{triv}\left(\lambda\right)}{d\lambda}|_{\lambda=\lambda_*}\neq 0$ and $\beta_{triv}\left(\lambda_*\right)=0$, the trivial solution is stable at one side of $\lambda_*$ and is unstable at another side of $\lambda_*$.
Based on the arguments above, the nontrivial solution $x(s)$ is called to be formally stable if $\beta(s)<0$ and unstable if $\beta(s)>0$.
Finally, we quote below the secondary bifurcation theorem that forms the basis of the existence of ripples in Section 4 (cf. \cite{Shearer} for a detailed discussion). The version that we use comes from \cite{MartinM}.

\begin{theorem}(Secondary bifurcation theorem) \label{thm6.4}
\emph{Let $X$ and $Y$ be real Banach spaces and assume $F: \mathbb{R}^2\times X\rightarrow Y$ is $C^{2}$ and there exist closed subspaces $X_{1}$ of $X$ and $Y_{1}$ of $Y$ such that $F:\mathbb{R}^2\times X_{1}\rightarrow Y_{1}$ and}

(H1) $F(\lambda,\beta,0)=0$ for all $(\lambda,\beta)\in \mathbb{R}^2$;

(H2) For some $(\lambda_*, \beta_*)\in \mathbb{R}^2$, $ F_{x}\left(\lambda_*, \beta_*, 0\right)$ is a Fredholm operator with $\mathcal{N} \left( F_{x}\left(\lambda_*, \beta_*, 0\right)\right)$ and $Y /\mathcal{R}\left(F_{x}\left(\lambda_*, \beta_*, 0\right)\right)$ are $1$-dimensional and
the null space generated by $x_1$, and the transversality condition
\begin{equation}
F_{\lambda,x}\left(\lambda_*,\beta_*, 0\right)[1,x_1]\not\in\mathcal{R}\left(F_{x}\left(\lambda_*,\beta_*, 0\right)\right)\nonumber
\end{equation}
holds;

(H3) There exist closed spaces $X_{2}, Y_{2}$ such that $X=X_{1}\oplus X_{2}$ and $Y=Y_{1}\oplus Y_{2}$;

(H4) $\mathcal{N}(F_{x}(\lambda_*,\beta_*,0))\cap X_{i}=span\{x_{i}\}$ for $i=1,2$;

(H5) There are $0\neq \psi_{i}\in Y$ with $span\{\psi_{i}\}\oplus\left(Y_{i}\cap \mathcal{R}(F_{x}(\lambda_*,\beta_*,0))\right)=Y_{i}$ for $i=1,2$;

(H6) $F_{x}(\lambda_*, \beta_*,0)X_{2}\subset Y_{2}$;

(H7) $F_{\lambda,x}(\lambda_*, \beta_*,0)[1,x_{i}]\not\in\mathcal{R}\left(F_{x}\left(\lambda_*,\beta_*, 0\right)\right)$ for $i=1,2$;

(H8) \begin{equation}  \nonumber\\
\left|
\begin{array}{cc}
\langle F_{\lambda x}(\lambda_*,\beta_*,0)[1,x_{1}]|\psi'_{1} \rangle  ~&~  \langle F_{\beta x}(\lambda_*,\beta_*,0)[1,x_{1}]|\psi'_{1} \rangle \\
~ & ~\\
\langle F_{\lambda x}(\lambda_*,\beta_*,0)[1,x_{2}]|\psi'_{2} \rangle   ~&~ \langle F_{\beta x}(\lambda_*,\beta_*,0)[1,x_{2}]|\psi'_{2} \rangle  \\
\end{array}
\right|\neq0
\end{equation}
and
\begin{equation}  \nonumber\\
\left|
\begin{array}{cc}
\langle F_{\lambda x}(\lambda_*,\beta_*,0)[1,x_{1}]|\psi'_{1} \rangle   ~&~  \langle F_{x x}(\lambda_*,\beta_*,0)[x_{1},x_{1}]|\psi'_{1} \rangle \\
~ & ~\\
\langle F_{\lambda x}(\lambda_*,\beta_*,0)[1,x_{2}]|\psi'_{2} \rangle   ~&~ \langle F_{x x}(\lambda_*,\beta_*,0)[x_{1},x_{2}]|\psi'_{2} \rangle  \\
\end{array}
\right|\neq0
\end{equation}
hold, where $\langle\psi_i |\psi'_{j}\rangle=\delta_{ij} $ and $\mathcal{N}(\psi'_{i})=\mathcal{R}(F_{x}(\lambda_*,\beta_*,0))$ for $i=1,2$ with $\langle\cdot|\cdot\rangle$ denoting the duality pairing on $Y\times Y'$.

Then, there exists an interval $\mathcal{I}$ containing zero and smooth functions $\chi,\kappa: \mathcal{I}\times \mathcal{I}\rightarrow \mathbb{R}$ such that $\kappa(0,0)\neq0$, and for $\beta=\beta_*+\theta$ with $0\neq\theta\in \mathcal{I}$, the curve
$$
\mathcal{S}_{\beta}:=\{ (\lambda(\beta-\beta_*,s),\beta,x(\beta-\beta_*,s)): s\in \mathcal{I}\}
$$
with
$$
\lambda(\theta,s):=\lambda_*+\theta\chi(\theta,s),
$$
$$
x(\theta,s):=\theta\kappa(\theta,s)x_{1}+\theta sx_{2}+z(\lambda_*+\theta\chi(\theta,s),\beta,\theta\kappa(\theta,s)x_{1}+\theta sx_{2})
$$
is a secondary branch of solutions intersecting the Primary branch $\mathcal{P}$ of bifurcation at $(\lambda(\theta,0),x(\theta,0))$ whereby $x(\theta,0)\neq 0$.
\end{theorem}

\section*{Acknowledgments}
Dai was supported by National Natural Science Foundation of China (No.
12371110). Xu was supported by the Postdoctoral Science Foundation of China (2023M731381). Zhang was supported by National Natural Science Foundation of
China (No. 12301133), the Postdoctoral Science Foundation of China (2023M741441) and Jiangsu Education Department (No. 23KJB110007).

\section*{Data Availability Statements}
Data sharing not applicable to this article as no datasets were generated or analysed during the current study.

\section*{Conflict of interest}
The authors declare that they have no conflict of interest.

\end{document}